\documentclass[12pt,reqno]{amsart}
\usepackage{amsthm}
\usepackage{amsfonts}
\usepackage{amsmath}
\usepackage{amssymb}
\usepackage{mathrsfs}
\usepackage{epsfig}
\usepackage{graphicx}
\usepackage{epstopdf}
\usepackage{tikz-cd}
\usepackage{color}
\usepackage{ifthen}
\usepackage{float}
\usepackage{cancel} 
\usepackage{comment}
\usepackage{hyperref}
\makeatletter
\@namedef{subjclassname@2010}{%
\textup{2010} Mathematics Subject Classification}
\makeatother
\usepackage[normalem]{ulem}
\newtheorem{theorem}{Theorem}[section]
\newtheorem{proposition}[theorem]{Proposition}
\newtheorem{corollary}[theorem]{Corollary}
\newtheorem{lemma}[theorem]{Lemma}

\theoremstyle{remark}
\newtheorem{remark}[theorem]{Remark}

\newtheorem{dummy}{Theorem} 
\newenvironment{realizationthm}[1][]
  {%
   \begin{dummy}[\bf Realization theorem]}
  {\end{dummy}}

 \newtheorem{dummyy}{Theorem} 
\newenvironment{interpolationthm}[1][]
  {%
   \begin{dummyy}[\bf Interpolation theorem]}
  {\end{dummyy}}
  
  \newtheorem{dummyyy}{Theorem} 
\newenvironment{Toeplitzcoronathm}[1][]
  {%
   \begin{dummyyy}[\bf Toeplitz corona theorem]}
  {\end{dummyyy}}
  
  \newtheorem{dummyyyy}{Theorem} 
\newenvironment{extensionthm}[1][]
  {%
   \begin{dummyyyy}[\bf Extension theorem]}
  {\end{dummyyyy}}

\theoremstyle{definition}
\newtheorem{definition}[theorem]{Definition}

\newcommand{\bq}{\begin{equation}}
\newcommand{\eq}{\end{equation}}
\newcommand{\beqn}{\begin{eqnarray*}}
\newcommand{\eeqn}{\end{eqnarray*}}
\newcommand{\beq}{\begin{eqnarray}}
\newcommand{\eeq}{\end{eqnarray}}

\usepackage{soul}
\newcommand{\rar}{\rightarrow}

\newcommand{\bc}{\begin{centre}}
\newcommand{\ec}{\end{centre}}

\newcommand{\SAE}{\mathcal S \mathcal A(\mathbb E)} \newcommand{\AKE}{\mathcal A\mathcal K (\mathbb E)}
\newcommand{\UCE}{\mathcal U \mathcal C(\mathbb E)}
\newcommand{\CDF}{C(\overline{\mathbb D})_F^+}
\newcommand{\CDE}{C(\overline{\mathbb D})_{\mathbb E}^+}

\newcommand{\PKE}{\mathbb {C}_{\mathbb E}^+}
\newcommand{\SAFE}{\mathcal{SA}_f(\mathbb E)}

\newcommand{\ba}{\begin{array}}
\newcommand{\ea}{\end{array}}

\newcommand{\inp}[2]{\langle{#1},\,{#2} \rangle}

\renewcommand{\Delta}{{\nabla}}

\newcommand*{\Ge}{\geqslant}

\newcommand*{\Le}{\leqslant}

\begin{document}
\title[Realization, interpolation, Toeplitz corona and extension]{Function Theory on Tetrablock: Realization, Interpolation, Extension and Toeplitz Corona Theorem}
\author[S. Jain]{Shubham Jain}
\address [S. Jain]{Department of Mathematics, Indian Institute of Technology Guwahati, Guwahati 781039, India}
\email{shubjainiitg@iitg.ac.in, shubhamjain4343@gmail.com}
\author[S. Kumar ]{Surjit Kumar}
\address[S. Kumar]{Department of Mathematics, Indian Institute of Technology Madras, Chennai 600036, India} 
\email{surjit@iitm.ac.in}
\author[M. K. Mal]{Milan Kumar Mal}
\address[M. K. Mal]{Department of Mathematics, Indian Institute of Technology Madras, Chennai 600036, India}
 \email{ma21d018@smail.iitm.ac.in; milanmal1702@gmail.com }
\author[P. Pramanick]{Paramita Pramanick}
\address[P. Pramanick]{Statistics and Mathematics Unit, Indian Statistical Institute Kolkata, Kolkata 700108, India}
\email{paramitapramanick@gmail.com}



\date{}

\begin{abstract}
We introduce a Schur–Agler type class associated with the tetrablock and establish a realization theorem for this class. Furthermore, we provide a tetrablock analog of the interpolation theorem, extension theorem, and the Toeplitz corona theorem.
\end{abstract}

\maketitle

\section{Introduction}
The classical {\it Schur class} of the unit disc $\mathbb D$, denoted by $\mathcal{S}(\mathbb D),$ consists of the holomorphic functions $f$ which satisfy $\displaystyle \sup_{z \in \mathbb D}|f(z)| \leq 1.$ The corresponding realization theorem states that $f\in \mathcal{S}(\mathbb D)$ if and only if there exist a Hilbert space $\mathfrak{H}$ and 
a unitary operator $$ V=\begin{bmatrix}
    A & B\\ C& D
\end{bmatrix}: {\substack{\mathbb C\\ \oplus\\ \mathfrak{H}}} \to {\substack{\mathbb C\\ \oplus\\ \mathfrak{H}}}$$ such that $f(z)=A +zB (I_\mathfrak{H}-zD)^{-1}C.$

Agler generalized this to the bidisc in \cite{AJ1990} (also, see \cite[Theorem 11.13]{AM2002}). A realization theorem on certain general domains $\Omega$ has also been obtained in  \cite{AJ1990} for {\it Schur-Agler class}, a subclass of  Schur class $\mathcal{S}(\Omega)=\{f \in \mathrm{Hol}(\Omega) :\sup_{z \in \Omega}|f(z)|\leq 1\}.$ 
The case when the domain $\Omega$ is the unit disc or the unit bidisc or the symmetrized bidisc, the Schur-Agler class coincides with the Schur class (see \cite{BS2018, AY2017, AM2002, AM2020, AY1999} ). 
For the domains of the type $\mathcal{D}_{\bf Q}=\{ z\in \mathbb C^n: \|{\bf Q}(z)\|<1\}$, a realization formula for the Schur-Agler class is given in \cite{BB2004, BMV2018}, where ${\bf Q}(z)$ is a matrix-valued polynomial in $n$ complex variables (refer also to \cite{AT2003}). 
A class of test functions has been used in \cite{DMM2007, DM2007} to develop a generalization of the realization theorem for a wider class of domains, including the annulus.
A standard application of the realization formula is a Pick interpolation type theorem (see \cite[Theorem 1.10]{BB2004}, \cite[Theorem 5.1]{AY2017}, \cite[page. 508]{BS2018}, \cite[Chapter 11]{AM2002}). 
Furthermore, a vector-valued generalization of the realization theorem has been carried out on various domains. Subsequently, an appropriate version of the Toeplitz corona theorem has been obtained (see \cite[Theorem 1.3]{BS2022}, \cite[Theorem 5.1]{BB2004}, \cite[Theorem 11.57]{AM2002}).

In this note, complex separable Hilbert spaces are denoted by $\mathfrak{H}, \mathfrak{K}, \mathfrak{L}$ etc. Let $\mathcal{B}(\mathfrak{K}) $ denote the algebra of bounded linear operators on $\mathfrak{K.}$ By a commuting $3$-tuple $T=(T_1, T_2, T_3)$ on $\mathfrak{K},$ we mean that $T_1, T_2, T_3$ are mutually commuting bounded linear operators on $\mathfrak{K}.$ The Hilbert space adjoint of the commuting $3$-tuple $T=(T_1,T_2,T_3)$ is denoted by $T^*=(T_1^*, T_2^*, T_3^*).$ Let $\sigma(T)$ and $\sigma_\pi(T)$ denote the Taylor joint spectrum and the approximate point spectrum of a commuting tuple $T$, respectively.
If $Y$ is a subset of $\mathfrak{H}$, then the closed linear span of $Y$ is denoted by $\bigvee\{x: x\in Y\}.$ We use $A^t$ to denote the transpose of a matrix $A,$ and write $A\geq 0$ to indicate the positive semi-definiteness of $A.$ 
Let $X$ be a non-empty set. For any function $f:X\to \mathbb C,$ $\|f\|_\infty$ denote the supremum norm of $f.$ A function $g:X \times X \to \mathbb C$ is said to be positive semi-definite if for any $n\in \mathbb N$, $x_1,\ldots,x_n\in X,$ and scalars $c_1,\ldots,c_n,$ we have $\sum_{i,j=1}^n \bar{c}_ic_j g(x_i,x_j) \geq 0.$ The notation $g\succcurlyeq 0$ on $X$ is used to indicate that a function $g:X \times X \rar \mathbb C $ is positive semi-definite.
 For $g_1, g_2 : X \times X \to \mathbb{C},$ by the (Schur) product $g_1g_2,$ we mean the point-wise product of $g_1$ and $g_2,$ that is, $g_1g_2(z,w)=g_1(z,w)g_2(z,w).$
The closure of a domain $\Omega$ is denoted by $\overline{\Omega}.$  The notation $\rm{Hol}(\Omega)$ stands for the set of all holomorphic functions on a domain $\Omega.$ The space of complex-valued continuous functions on $\overline{\Omega}$ is denoted by $C(\overline {\Omega})$ and $A(\overline{\Omega})$ stands for $\mathrm{Hol}(\Omega) \cap C(\overline {\Omega}).$ The set of all complex-valued bounded holomorphic functions on a domain $\Omega$ is denoted by $H^\infty(\Omega)$.
For $s>0, \; z=(z^{(1)}, z^{(2)}, z^{(3)}) \in \mathbb C^3,$ we denote $s \cdot z=(sz^{(1)}, sz^{(2)}, s^2z^{(3)}).$  

The {\it tetrablock} is the domain \beq \label{tetra-new} \mathbb E = \{z \in \mathbb C^3 : 1-z^{(1)}\alpha-z^{(2)}\beta+z^{(3)}\alpha \beta \neq 0~\mbox{for all~} \alpha,\beta \in \overline{\mathbb D}\}.
\eeq
There has been considerable work on the tetrablock, examined from both geometric and operator-theoretic perspectives, for instance, see \cite{AWY2007, BT2014}.
One of the characterizations of the tetrablock obtained in \cite[Theorem 2.2]{AWY2007}, is as follows: 
\beq\label{domain description}
z \in \mathbb E \mbox{ if and only if } |z^{(2)}|<1 \mbox{ and } |\psi_{\alpha}(z)|<1 \mbox{ for all } \alpha \in \overline{\mathbb D} ,
\eeq 
where $\psi_\alpha(z)=\frac{\alpha z^{(3)}-z^{(1)}}{\alpha z^{(2)}-1}.$
 For any $z \in \mathbb E$, let $E(z):\overline{\mathbb D}\to \mathbb C$ be the function given by
 \beq \label{E} E(z)(\alpha)=\frac{\alpha z^{(3)}-z^{(1)}}{\alpha z^{(2)}-1}.\eeq
Note that $E(z)$ is a continuous function on $\overline{\mathbb D}$ and $\|E(z)\|_{\infty} <1$ for any $z \in \mathbb E$. 

 We now define a class $\mathcal M$ of commuting $3$-tuples $T=(T_1, T_2, T_3)$ satisfying $$\|T_2\| <1 \mbox{ and } \|\psi_{\alpha}(T)\|=\|(\alpha T_3-T_1)(\alpha T_2-I)^{-1}\|<1\; \text{for all } \alpha \in \overline{\mathbb D}.$$

It is straightforward to verify that if $T \in \mathcal M$ then
    \begin{itemize}
        \item[(i)] $T^* \in \mathcal M.$ Indeed, $\|T_2\|<1$ if and only if $\|T_2^*\|<1.$ Since $\overline {\mathbb D}$ is closed under conjugation, we have $\|\psi_\alpha(T)\|<1$ if and only if $\|\psi_{\alpha}(T^*)\|<1$ for each $\alpha \in \overline {\mathbb D}.$
        \item [(ii)] $r \cdot T:=(rT_1,rT_2,r^2 T_3) \in \mathcal M$ for any $0 < r \leq 1.$
    \end{itemize}   
For clarity, we occasionally denote $(T, \mathfrak{K}) \in \mathcal{M}$ instead of $T\in \mathcal{M},$ to stress that the operator tuple $T=(T_1,T_2,T_3)$ is defined on $\mathfrak{K}.$ 
We now introduce the Schur-Agler type class associated with the tetrablock $\mathbb E$ (cf. \cite[page 50]{BB2004}), which requires that  $\sigma(T) \subseteq \mathbb E$ for any $T \in \mathcal M$. This has been observed in Lemma \ref{spectrum}. 
\begin{definition} \label{Schur-Agler class}
The\textit{ Schur-Agler type class} for the tetrablock $\mathbb E$, denoted by $\mathcal S \mathcal A(\mathbb E),$  is defined as 
  \beqn 
\mathcal S \mathcal A(\mathbb E):=\left\{f \in \mathrm{Hol}(\mathbb E) : \|f(T)\| \leq1 \mbox{ for all } T \in \mathcal M \right\}.
\eeqn  
\end{definition}

A weak kernel $k$ on $\mathbb E$ is defined as a function $k: \mathbb E \times \mathbb E \rightarrow \mathbb C$ that is positive semi-definite. If, in addition, $k(z, z) \neq 0$ for all $z \in \mathbb E,$ we say that $k$ is a kernel.
Let  $F $ be any subset of $\mathbb E$. A function $\Gamma: F \times F \to C(\overline{\mathbb D})^*$ is said to be {\it positive kernel with values in $C(\overline{\mathbb D})^*$} if it satisfies the following property: for any $n \in \mathbb N$, and for any $z_1,\ldots,z_n \in F, \;\; c_1,\ldots,c_n \in \mathbb C$, and $f_1,\ldots,f_n \in C(\overline{\mathbb D}),$ we have 
$$\sum_{i,j=1}^n \bar{c}_i c_j \Gamma(z_i,z_j)(f_i \bar{f}_j) \geq 0.$$ 
The set of all positive kernels with values in $C(\overline{\mathbb D})^*$ is denoted by $C(\overline{\mathbb D})_F^+$, and the set of all positive semi-definite functions $\Delta: F \times F \to \mathbb C$ is denoted by $\mathbb C_F^+.$

The definition of an admissible kernel given below is motivated by \cite[Definition 11.1]{AM2002}. 
\begin{definition} 
A kernel $k: \mathbb E \times \mathbb E \rightarrow \mathbb C$ is said to be \textit{admissible} if $(1-z^{(2)}\overline{w}^{(2)})k(z,w) \succcurlyeq 0$ and $(1 - \psi_\alpha (z) \overline{\psi_\alpha (w)})k(z, w) \succcurlyeq 0$ for all $ \alpha \in \overline{\mathbb D}.$
Similarly, we define {\it weak admissibility} by replacing the kernel with a weak kernel.
\end{definition}
The class of admissible kernels on $\mathbb E$ is denoted by $\mathcal A\mathcal K (\mathbb E).$ 

A notion of Hardy space on a class of domains consisting tetrablock can be found in \cite{GR2024}. The Hardy space $H^2(\mathbb{E})$ on the tetrablock $ \mathbb{E}$ was recently studied in \cite[Section 3.2]{BCJ2024}, following the approach in \cite{MRZ2013}. A key idea in this approach is to realize that $\mathbb{E}$ is the image of a type-II classical Cartan domain in three dimensions under a proper holomorphic map. 
To be more precise, there exists a unitary map between the Hardy space $H^2(\mathbb{E})$ and a closed subspace of the Hardy space of a type-II Cartan domain (see \cite[Theorem~1.3]{BCJ2024}). The multiplier algebra of $H^2(\mathbb{E})$ coincides with $H^\infty(\mathbb{E})$ (see \cite[p.~1]{HM1969}). In particular, for every $\alpha \in \overline{\mathbb{D}}$, the function $\psi_\alpha$ and $z^{(2)}$ acts as a contractive multiplier on $H^2(\mathbb{E})$. This combined with \cite[Theorem~5.21]{PR2016} yields that the Szeg\"o  kernel on $\mathbb E$, the reproducing kernel for $H^2(\mathbb{E}),$ as described in \cite[Eqn.~(9)]{BCJ2024}, is an admissible kernel.

In this paper, we introduce a Schur–Agler type class associated with the tetrablock (see Definition \ref{Schur-Agler class}) and establish a realization theorem (Theorem \ref{realization theorem}) for this class. Furthermore, we present an appropriate version of the Nevanlinna–Pick interpolation theorem (Theorem \ref{interpolation theorem}), extension theorem (Theorem \ref{extension theorem}), and the Toeplitz corona theorem (Theorem \ref{Toeplitzcoronatheorem1}). 

In Section \ref{S2}, we compile the necessary results to establish Theorem \ref{realization theorem}. 
Section \ref{S3} is devoted to prove one of the main results of the paper, the realization theorem: 
\begin{realizationthm}\label{realization theorem}
Let $f :\mathbb E\to \mathbb C$ be a function. Then the following statements are equivalent:
    \begin{itemize}
        \item [$\mathrm{(i)}$] $f \in \SAE.$
        \item [$\mathrm{(ii)}$]  $(1 - f(z) \overline{f (w)})k(z, w) \succcurlyeq 0$ on $\mathbb E$ for each $k \in \AKE.$
        \item [$\mathrm{(iii)}$] There exist a positive kernel $\Gamma: \mathbb E \times \mathbb E \rightarrow C(\overline{\mathbb D})^*$ and a weak kernel $\Delta: \mathbb E \times \mathbb E \to \mathbb C$ such that for each $z,w \in \mathbb E,$
        \begin{equation}\label{keyequation1}
         1 - f(z) \overline{f (w)}=\Gamma(z, w)\Big(1- E(z)\overline{E(w)}\Big)+ (1-z^{(2)}\overline{w}^{(2)})\Delta(z,w),
        \end{equation} where $E(z)$ is as in \eqref{E}.
         \item [$\mathrm{(iv)}$] There are Hilbert spaces $\mathfrak {H}_1,\;\mathfrak{H}_2,$ a unital $*$-representation $\rho :C(\overline{\mathbb D}) \rightarrow \mathcal{B}(\mathfrak{H}_1)$ and a unitary $V : \mathbb C \oplus \mathfrak{H} \rightarrow \mathbb C \oplus \mathfrak{H}$, $\mathfrak{H}=\mathfrak{H}_1\oplus \mathfrak{H}_2,$ such that $$V=\begin{bmatrix}
A & B \\
C & D
\end{bmatrix},\; \text{and}$$ 
$$f(z)=A +B\begin{bmatrix} 
    \rho(E(z)) & 0\\ 0& z^{(2)}I_{\mathfrak H_2} 
\end{bmatrix}\left(I_{\mathfrak{H}}-D\begin{bmatrix} 
    \rho(E(z)) & 0\\ 0& z^{(2)}I_{\mathfrak H_2} 
\end{bmatrix}\right)^{-1}C.$$
    \end{itemize}
\end{realizationthm}
The classical interpolation theorem on the unit disc states that, given points \( \lambda_1, \dots, \lambda_n \in \mathbb{D} \) and \( w_1, \dots, w_n \in \overline{\mathbb{D}} \), there exists a function $f \in \mathcal S(\mathbb D)$ such that \( f(\lambda_i) = w_i \) for \( 1 \leq i \leq n \) if and only if the associated Pick matrix
\beq P=\label{basic pick matrix}
\begin{bmatrix} \frac{1 - w_i \overline{w}_j}{1 - \lambda_i \overline{\lambda}_j} \end{bmatrix}_{i,j=1}^n \geq 0,
\eeq
that is, $P$ is positive semi-definite (see \cite[Theorem 1.3]{AM2002}). In Subsection \ref{ss4.1}, we establish a suitable analogue of this result in the setting of the tetrablock as an elegant consequence of the realization theorem. 
\begin{interpolationthm} \label{interpolation theorem}
  Let $z_1, \ldots, z_n \in \mathbb E$ and $w_1, \ldots, w_n \in \overline{\mathbb D}$. Then the following statements are equivalent:
    \begin{itemize}
        \item [(i)] There exists a function $f \in \mathcal{SA}(\mathbb E)$ such that $f(z_i)=w_i$ for all $i=1,\ldots,n.$
        \item [(ii)] For all $k \in \AKE$, we have 
        $$\begin{bmatrix}
            (1-w_i\overline{w}_j)k(z_i,z_j)
        \end{bmatrix}_{i,j=1}^n \geq 0.$$
        \item [(iii)] Let $F=\{z_1,\ldots z_n\}.$ There exist a positive kernel $\Gamma \in  C(\overline{\mathbb D})_F^+$ and a weak kernel $\Delta \in \mathbb C_F^+$ such that  for all $i,j=1,\ldots,n,$
        \beq \label{interpolationeq1} 1-w_i\overline{w}_j=\Gamma(z_i,z_j)(1-E(z_i)\overline{E(z_j)})+ (1-z_i^{(2)}\bar{z}_j^{(2)})\Delta(z_i,z_j).
        \eeq 
    \end{itemize}
\end{interpolationthm}
Let $V$ be a subset of $\mathbb E$. Denote $\mathcal H \mathcal E(V),$ the collection of all bounded functions on $V$ which has an extension to a holomorphic function in a neighbourhood of $V$. We say that a commuting $3$-tuple $T$ is subordinate to $V$, if $\sigma(T) \subseteq V$ and $g(T)=0$ for every function $g$ that is holomorphic in a neighbourhood of $V$ and $g|_{V}=0.$ 
If $f \in \mathcal H \mathcal E(V)$ and $T$ is subordinate to $V$, then $f(T)$ can be defined as $f(T)=g(T),$ where $g$ is any extension of $f$ in a neighbourhood of $V.$

Subsection \ref{S6} is devoted to the extension theorem, that is, Theorem \ref{extension theorem}. 
Before stating the extension theorem, we introduce the following notation. Let $\mathcal{M}'$ be the collection of all commuting $3$-tuples $T=(T_1, T_2, T_3)$ with $\sigma(T)\subseteq \mathbb E$ such that  $$\|T_2\| \leq 1 \mbox{ and } \|\psi_{\alpha}(T)\|\leq1\; \text{for all } \alpha \in \overline{\mathbb D}.$$
\begin{extensionthm}\label{extension theorem}
    Let $V \subseteq \mathbb E$ and let $f \in \mathcal H \mathcal E(V).$ Then there exists $g \in H^\infty(\mathbb E)$ such that $\frac{1}{\|f\|_{\infty}}g\in \mathcal{SA}(\mathbb E)$, $g|_{V}=f$, and $\|g\|_\infty=\|f\|_{\infty}$ if and only if \beq \label{extension condition} \|f(T)\| \leq \|f\|_{\infty},\; \text{ for any } T \in \mathcal{M}'\; \text{ subordinate to } V.\eeq 
\end{extensionthm}
 A vector-valued analog of Schur-Agler type class is defined below (cf. \cite{BB2004}).
\begin{definition}
    Let $\mathfrak{L}_1, \mathfrak{L}_2$ be two Hilbert spaces. The $\mathcal{B}(\mathfrak{L}_1,\mathfrak{L}_2)$-valued Schur-Agler type class, denoted by $\mathcal{SA}_{\mathbb E}(\mathfrak{L}_1,\mathfrak{L}_2)$, is the collection of all holomorphic functions $f: \mathbb E\to \mathcal{B}(\mathfrak{L}_1,\mathfrak{L}_2) $ such that $\|f(T)\| \leq 1$ for all $T \in \mathcal{M}$.
\end{definition}
In section \ref{S5}, we present a vector-valued version of Theorem \ref{realization theorem} for the class $\mathcal{SA}_{\mathbb E}(\mathfrak{L}_1,\mathfrak{L}_2)$ (see Theorem \ref{realization theorem vector}). As an application, we obtain the Toeplitz corona theorem:
\begin{Toeplitzcoronathm} \label{Toeplitzcoronatheorem1}
    Let $\varphi_1,\ldots,\varphi_n \in H^\infty(\mathbb E)$ and let $\delta$ be a positive real number. Then the following are equivalent:
    \begin{itemize}
        \item [(i)] There exists $f_1,\ldots,f_n\in H^\infty(\mathbb E)$ such that $(f_1,\ldots,f_n)^t \in \frac{1}{\delta}\mathcal{SA}_{\mathbb E}(\mathbb C,\mathbb C^n)$ and $$\sum_{j=1}^n\varphi_j(z)f_j(z)=1\;\; \text{ for all } z \in \mathbb E.$$
        \item [(ii)] The function $\big(\sum_{j=1}^n \varphi_j(z)\overline{\varphi_j(w)}-\delta^2\big)k(z,w)$ is positive semi-definite for all $k\in \AKE.$
        \item [(iii)] There exist a positive kernel $\Gamma:\mathbb E \times \mathbb E\to C(\overline{\mathbb D})^*$ and a weak kernel $\Delta: \mathbb E \times \mathbb E \to \mathbb C$ such that for all $z,w \in \mathbb E$, $$\sum_{j=1}^n \varphi_j(z)\overline{\varphi_j(w)}-\delta^2=\Gamma(z,w)(1-E(z)\overline{E(w)})+(1-z^{(2)}\overline{w}^{(2)})\Delta(z,w).$$
    \end{itemize}
\end{Toeplitzcoronathm}

\section{Preparatory results to prove the realization theorem} \label{S2}

In this section, we first prove an analogous result obtained in \cite[Proposition 3.3]{DM2007}( see also \cite{DMM2007}) in the context of the tetrablock. Though the idea of the proof is borrowed from \cite[Theorem 2.53]{AM2002}, for the sake of completeness, we provide its proof.

\begin{proposition}\label{decompositionresult}
    Let $\Gamma \in \CDF.$ Then there exist a Hilbert space $\mathfrak{H}$ and a function $L:F \to \mathcal{B}(C(\overline{\mathbb D}),\mathfrak{H})$ such that 
    \beq \label{kerneldecomposition} \Gamma(z,w)(f\bar{h})=\inp{L(z)f}{L(w)h}_{\mathfrak{H}}\; \text{for all } f,h \in C(\overline{\mathbb D}) \text{ and } z,w \in F.
    \eeq
    Moreover, there exists a unital $*$-representation $\rho:C(\overline {\mathbb D}) \to \mathcal{B}(\mathfrak{H})$ such that $L(z)(fh)=\rho(f)L(z)(h)$ for all $z \in F$ and $f, h \in  C(\overline{\mathbb D}).$ 
\end{proposition}
\begin{proof}
    Consider the map $\Gamma':(F\times C(\overline{\mathbb D})) \times (F\times C(\overline{\mathbb D})) \to \mathbb C$, defined by $\Gamma '((z,h_1),(w,h_2))=\Gamma(z,w)(h_1\bar{h}_2).$ 
  A routine verification shows that $\Gamma'$ is a positive semi-definite function on $F\times C(\overline{\mathbb D}).$ By \cite[Theorem 2.53]{AM2002}, there exist a Hilbert space $\mathfrak{H}$ and a function $g: F \times C(\overline{\mathbb D}) \rar \mathfrak{H}$ such that $$\Gamma'((z,h_1),(w,h_2))=\inp{g(z,h_1)}{g(w,h_2)}_{\mathfrak{H}}.$$
  In fact, the Hilbert space $\mathfrak H$ can be chosen such that $\mathfrak{H}=\bigvee\{g(z,h): z\in F, h\in C(\overline{\mathbb D})\}$.
Let  $L: F \to \mathcal{B}(C(\overline{\mathbb D}),\mathfrak{H})$ be a map given by $L(z)(h)=g(z,h)$ for any $z\in F$ and $h \in C(\overline{\mathbb D}).$ Note that $L$ is well defined. Indeed, for any $z\in F$, $$\|L(z)h\|^2=\|g(z,h)\|^2=\Gamma(z,z)(h\bar{h}) \leq \|\Gamma(z,z)\| \|h\|_{\infty}^2.$$ 
Consider $\rho:C(\overline{\mathbb D})\to \mathcal{B}(\mathfrak{H})$, defined by $\rho(h_1)g(z,h_2)=g(z,h_1h_2)$. It is easy to check that $\rho$ is a unital $*$-representation and 
$\rho(h_1)L(z)(h_2)=L(z)(h_1h_2).$ This completes the proof.
\end{proof}

Let $F$ be a finite subset of $\mathbb E$ and $|F|$ be the cardinality of $F.$ Let $\mathcal{C}_F$ be the collection  of $|F| \times |F|$ matrices of the form $$
      \begin{bmatrix}\Gamma(z,w)(1-E(z)\overline{E(w)})+(1-z^{(2)}\overline{w}^{(2)})\Delta(z,w)\end{bmatrix}_{z,w \in F},$$ where $\Gamma \in C(\overline{\mathbb D})_F^+$ and $ \Delta \in \mathbb C_F^+.$ Following the 
approach in \cite[Lemma 3.4]{DM2007}, we observe that $\mathcal C_F$ is a closed cone.
Also, for any positive semi-definite $|F|\times |F|$ matrix $P=\begin{bmatrix}
    p(z,w)
\end{bmatrix}_{z,w \in F}$, if we consider $\Delta(z,w)=P(z,w)\Delta_1(z,w)$, where $\Delta_1(z,w)=\frac{1}{1-z^{(2)}\overline{w}^{(2)}}$, then  
$$\begin{bmatrix}P(z,w)\end{bmatrix}_{z,w\in F}=\begin{bmatrix}(1-z^{(2)}\overline{w}^{(2)})\Delta(z,w)\end{bmatrix}_{z,w\in F} \in \mathcal{C}_F\; \text{with } \Gamma=0.$$
This shows that $\mathcal{C}_F$ contains all positive semi-definite $|F| \times |F|$ matrices, and hence has nonempty interior.
Clearly, the $|F| \times |F|$ matrix $[1]$ with all entries equal to $1$ is in $\mathcal{C}_F$ and 
 the matrix $[c_i\bar{c}_j]_{i, j=1}^n \in \mathcal{C}_F$ for all scalars $c_1, \ldots c_n,$ where $n = |F|.$

For a nonempty set $X,$ a function $g: X \times X \to \mathbb C$ is said to be self-adjoint on $X$ if $g(z,w)=\overline{g(w,z)}$ for all $z, w \in X.$
The preceding discussion is used to prove the following result (cf. \cite[Theorem 11.5]{AM2002}), which characterizes self-adjoint functions on $\mathbb E$  whose product with every admissible kernel is positive semi-definite.
\begin{theorem}\label{keytheorem}
    Let $g: \mathbb E \times \mathbb E \to \mathbb C$ be a self-adjoint function. Then   $gk \succcurlyeq 0$ for all $k \in \AKE$ if and only if there exist $\Gamma \in \CDE$ and $\Delta \in \PKE$ such that for all $z,w \in \mathbb E,$
    \beq \label{keyequation}
    g(z,w)=\Gamma(z, w)(1- E(z)\overline{E(w)})+ (1-z^{(2)}\overline{w}^{(2)})\Delta(z,w).\eeq
\end{theorem}

\begin{proof}  
Let $F$ be any finite subset of $\mathbb E.$
First, we show that \eqref{keyequation} holds for every $z, w \in F$.
\textbf{Claim:} $[g(z,w)]_{z,w\in F} \in \mathcal{C}_F$.
 Suppose $G=[g(z,w)]_{z,w\in F} \notin \mathcal{C}_F$. Since $\mathcal{C}_F$ is a closed cone with nonempty interior, by an application of Hahn-Banach separation theorem (see \cite[Theorem 3.4]{RD1991}), 
 there exists a linear functional $L$ on the space of $|F| \times |F|$ matrices such that $L(G)<0$ and $L(M) \geq 0$ for all $M\in \mathcal{C}_F.$ The functional $L$ can be chosen such that for all self-adjoint matrices $A,$ $L(A)=tr(AC)$, where $C$ is a fixed self-adjoint matrix.

Let $F = \{z_1, \ldots, z_n\}$. For any choice of scalars $c_1, \ldots, c_n$, we have
$$
 \sum_{i,j=1}^n \bar{c}_i c_j C_{j,i} = L(M) \geq 0,
$$
where $M = [\bar{c}_i c_j]_{i,j=1}^n \in \mathcal{C}_F$. Since every principal submatrix of a positive semi-definite matrix is itself positive semi-definite, it follows that the function $C^t : F \times F \to \mathbb{C}$, defined by $C^t(z_i, z_j) = C_{j,i}$, is positive semi-definite. Hence, $C^t$ is a weak kernel on $F$.
Moreover, by using $L(M) \Ge 0$ on $\mathcal{C}_F$ to the matrices of the form
$$
M = \left[\bar{c}_i (1 - z_i^{(2)} \overline{z}_j^{(2)}) c_j\right]_{i,j=1}^n \quad \text{and} \quad M_\alpha = \left[\bar{c}_i c_j (1 - \psi_\alpha(z_i) \overline{\psi_\alpha(z_j)})\right]_{i,j=1}^n, \quad \alpha \in \overline {\mathbb{D}},
$$
we conclude that $C^t$ is a weakly admissible kernel on $F$.
Now we extend $C^t$ to a weakly admissible kernel on $\mathbb E$ by defining $C^t(z, w)=0$ when $(z, w) \notin F \times F.$
For $\epsilon>0,$ $\epsilon k+ C^t \in \AKE$ for any $k \in \AKE.$ Since $g  k' \succcurlyeq 0$ for each $k' \in \AKE,$ $g (\epsilon k+ C^t) \succcurlyeq 0$ on $\mathbb E$ and consequently $g (\epsilon k+ C^t) \succcurlyeq 0$ on $F$ for any $\epsilon>0.$ Therefore, $g  C^t \succcurlyeq 0$ on $F.$ In particular, $\sum_{i,j=1}^n g(z_i,z_j) C^t(z_i,z_j) \geq 0.$ This yields that 
$$L(G)=tr(GC)=\sum_{i,j=1}^n G_{i,j} C_{j,i}=\sum_{i,j=1}^n  g(z_i,z_j) C^t(z_i,z_j) \geq 0.$$ This is a contradiction to the fact that $L(G)<0.$ Hence $[g(z,w)]_{z,w\in F} \in \mathcal{C}_F$. Therefore, there exist $\Gamma \in C(\overline{\mathbb D})_F^+$ and $\Delta \in \mathbb C_F^+$ such that for all $z,w\in F,$ $$g(z,w)=\Gamma(z,w)(1-E(z)\overline{E(w)})+ (1-z^{(2)}\overline{w}^{(2)})\Delta(z,w).$$

For any subset $F\subseteq \mathbb E$, consider the product topology on $C(\overline{\mathbb D})_F^+ \times \mathbb C_F^+$ induced by the weak-$*$ topology of $C(\overline{\mathbb D})_F^+\approx C(\overline{\mathbb D})^* \otimes \mathbb C^{|F| \times |F|}$ and the topology of $\mathbb C_F^+$. Define
$$ Y_F:=\left\{(\Gamma_F, \Delta_F) \in C(\overline{\mathbb D})_F^
    + \times \mathbb C_F^+: g(z,w) =\;\;\substack{\Gamma_F (z,w)\left(1-E(z)\overline{E(w)}\right)\\
 +\left(1-z^{(2)}\overline{w}^{(2)}\right)\Delta_F(z,w)}, \; z,w\in F\right\}.$$
A routine verification shows that the set $Y_F$ is compact. Let $\mathcal{A}$ be the directed set consisting of all finite subsets $F$ of $\mathbb E$ with the partial order $F_1 \leq F_2$ if $F_1 \subseteq F_2$ for any $F_1,F_2 \in \mathcal{A}$. Define $\pi_{F_1}^{F_2}: Y_{F_2} \to Y_{F_1}$ by $$\pi_{F_1}^{F_2} ((\Gamma_{F_2},\Delta_{F_2}))=(\Gamma_{F_2}|_{_{F_1}},\Delta_{F_2}|_{_{F_1}}).$$
Therefore, the triple $(Y_{F},\pi_{G}^{F},\mathcal{A})$ is an inverse limit of nonempty compact spaces, where $F, G \in \mathcal{A}.$
We now complete the proof by using Kurosh's theorem \cite[page 75]{AP1990}. 
\end{proof}
Recall that the approximate point spectrum of a commuting tuple $T=(T_1,\ldots,T_n)$ on $\mathfrak{H}$ is defined as the set of points $\lambda = (\lambda_1, \ldots, \lambda_n)\in \mathbb C^n$ such that there exists a sequence $\{x_m\} \subset \mathfrak{H}$ with $\|x_m\| = 1$ and $(T_j - \lambda_j)x_m \to 0$ as $m \to \infty$ for every $j = 1, \ldots, n$. It is shown in \cite{SZ1974} that $\sigma(T)$ is contained in the polynomially convex hull of $\sigma_{\pi}(T)$.

In view of Definition \ref{Schur-Agler class}, we show that $\sigma(T)\subseteq \mathbb E$ for any $T\in \mathcal{M}.$ This ensures that the functional calculus $f(T)$ is defined for any $f \in \mathrm{Hol}(\mathbb E).$

\begin{lemma}\label{spectrum}
    If $T \in \mathcal M$ then $\sigma(T)$ is contained in $\mathbb E.$
\end{lemma}
\begin{proof}
    Let $z=(z^{(1)},z^{(2)},z^{(3)}) \in \sigma_{\pi}(T).$ Then there exists a sequence $\{x_m\}$ of unit vectors in $\mathfrak H$ such that for each $j=1,2,3$, $(T_j-z^{(j)})x_m$ goes to $0$ as $m$ tends to $\infty.$ It follows from $\|T_2\| <1$ that $|z^{(2)}|<1.$ Since $\|(\alpha T_3-T_1)(\alpha T_2-I)^{-1}\|<1,$ for each $m,$ we have 
    $$\|(\alpha T_3-T_1)x_m\|^2 < \|(\alpha T_2-I)x_m\|^2.$$ Now taking $m \rightarrow \infty$ yields that $|\psi_\alpha(z)|<1.$ Thus $\sigma_\pi(T) \subset \mathbb E.$ Since $\mathbb E$ is polynomially convex, we have $\sigma(T) \subset \mathbb E.$
\end{proof}
 
We now define a unitary colligation in the context of tetrablock (cf. \cite{BT1998}).
  \begin{definition} \label{unitarycolligation}
       A function $f:\mathbb E \to \mathbb C$ is said to be associated to a {\it unitary colligation}, if there are Hilbert spaces $\mathfrak{H}_1,\; \mathfrak{H}_2,$ a unital $*$-representation $\rho:C(\overline{\mathbb D}) \to B(\mathfrak{H}_1)$ and a unitary $V$ given by  $$V=\begin{bmatrix}
A & B \\
C & D
\end{bmatrix}:\mathbb C \oplus \mathfrak{H} \rightarrow \mathbb C \oplus \mathfrak{H},$$
where $\mathfrak{H}=\mathfrak{H}_1\oplus \mathfrak{H}_2,$ and $$f(z)=A +B \begin{bmatrix}
    \rho(E(z)) & 0\\ 0 & z^{(2)}I_{\mathfrak{H}_2}
\end{bmatrix}\left(I_\mathfrak{H}-D\begin{bmatrix}
    \rho(E(z)) & 0\\ 0 & z^{(2)}I_{\mathfrak{H}_2}
\end{bmatrix}\right)^{-1}C.$$
  \end{definition}
 \begin{remark} \label{Operator D contraction} Here are a few quick observations: 
 \begin{itemize}
     \item [$\rm{(i)}$] Since $V$ is unitary, the operator $D$ is a contraction. 
     \item [$\rm{(ii)}$] For all $z \in \mathbb E,$ let $$X(z):=\begin{bmatrix}
    \rho(E(z)) & 0\\
    0 & z^{(2)}I_{\mathfrak{H}_2}
\end{bmatrix} \in \mathcal{B}(\mathfrak{H}).$$ 
Since $\rho$ is a unital $*$-representation and $\|E(z)\|_\infty < 1$ for all $z \in \mathbb{E}$, we obtain $\|\rho(E(z))\| < 1$ for all $z \in \mathbb{E}$. Therefore, it follows that $\|X(z)\| < 1$ for all $z \in \mathbb{E}$.
\end{itemize}
 \end{remark} 

Let $\mathcal U \mathcal C(\mathbb E)$ denote the collection of functions $f:\mathbb E \to \mathbb C$ which have an associated unitary colligation. We now establish some preliminary properties of elements in $\UCE$.

\begin{lemma}\label{properties of SA}  The following statements holds true.
    \begin{itemize}
    \item [$\rm {(i)}$] $\psi_\alpha \in \UCE$ for all $\alpha \in \overline{\mathbb D}.$
        \item [$\rm {(ii)}$] $\UCE \subseteq \mathcal{S}(\mathbb E).$
        \item [$\rm {(iii)}$]$f g \in \UCE$ for all $f, g \in \UCE.$
         \item [$\rm {(iv)}$] For a tuple $T\in \mathcal{M}$ defined on a Hilbert space $\mathfrak K, $ let $\pi: \UCE \to \mathcal {B}(\mathfrak{K})$ be a map given by $\pi(f)=f(T).$  Then $\pi (fg)=\pi(f)\pi(g).$ Moreover, if $\{f_n\}$ is a sequence in $\UCE$ which converges pointwise to $f\in \UCE$, then $\pi(f_n)$ converges to $\pi(f)$ in the weak operator topology. 
    \end{itemize}
\end{lemma}

\begin{proof} The proofs of (i), (ii), and (iii) follow directly from the Definition \ref{unitarycolligation}.
We now provide the proof of part $\rm{(iv)}.$
Let $(T,\mathfrak{K}) \in \mathcal{M}.$ Then by Lemma \ref{spectrum}, $\sigma(T) \subset \mathbb E$. Using the holomorphic functional calculus of $T,$ we define $\pi : \UCE \to \mathcal{B}(\mathfrak{K})$ as follows:  $$\pi(f)=f(T)=\frac{1}{(2\pi i)^3} \int_{\partial\Omega} M_T(z) f(z) dz,\; f\in \UCE$$ where $\Omega \subseteq \mathbb C^3$ is an open set containing $\sigma(T)$ with $C^1$-boundary, $\overline{\Omega} \subseteq \mathbb E$ and $M_T(z)$ is the Martinelli kernel corresponding to the tuple $T$ (see \cite[Chapter III, Proposition 11.1]{VF1982}). It is easy to verify that $\pi(fg)=\pi(f)\pi(g)$ for all $f, g\in \UCE.$
    Let $\{f_n\} \in \UCE$ be a sequence that converges pointwise to $f\in \UCE.$ For each $x,y\in \mathfrak{K}$, consider $d\mu(z):=\inp{M_T(z)x}{y}dz$ as a measure on the Borel subsets of $\partial\Omega$. 
   Since $\{f_n\}$ is uniformly bounded (see $\rm (ii)$), by dominated convergence theorem, we get \beqn \lim_{n \to \infty} \inp{f_n(T)x}{y} =\frac{1}{(2\pi i)^3} \int_{\partial\Omega}  \lim_{n\to \infty}f_n(z)d\mu(z)= \inp{f(T)x}{y}.\eeqn 
    This completes the proof.
    \end{proof}

A unital $*$-representation $\rho:C(\overline{\mathbb D}) \to B(\mathfrak{H})$ is called a {\it simple} representation if there are $\alpha_1, \ldots, \alpha_N \in \overline{\mathbb D}$ and orthogonal projections $P_1,\ldots,P_N\in \mathcal{B}(\mathfrak{H})$ with $\sum_{j=1}^N P_j=I_{\mathfrak{H}}$ such that $$\rho(h)=\sum_{j=1}^N P_j h(\alpha_j)\; \text{for all } h \in C(\overline{\mathbb D}).$$ 
In particular, $\rho(E(z))=\sum_{j=1}^N P_j \psi_{\alpha_j}(z),$ for any $z \in \mathbb E.$ Thus, the map $z \to \rho(E(z))$ is a $B(\mathfrak{H})$-valued holomorphic function on $\mathbb E$. 
Let $T \in \mathcal M$ be defined on $\mathfrak{K}.$ Then the holomorphic functional calculus is given by 
 $$\widetilde{\pi}(\rho(E)):=\rho(E(T))=\sum_{j=1}^N P_j \otimes \psi_{\alpha_j}(T)\in \mathcal B(\mathfrak H \otimes \mathfrak{K}).$$

The following result establishes a special case of the implication $(iv) \implies (i)$ in Theorem \ref{realization theorem} (cf. \cite[Lemma 3.1]{DM2007}).

\begin{lemma}\label{simplerep}
    Let $f \in \UCE$ and $\rho$ be the associated unital $*$-representation. If $\rho$ is a simple representation, then $f(T)$ is a contraction for any $T\in \mathcal{M}$.
\end{lemma}

\begin{proof}
Let $T=(T_1, T_2, T_3)\in \mathcal{M}$ be defined on a Hilbert space $\mathfrak{K}.$
Let $D$ and $X(z)$ be as given in the definition of $f\in \UCE$ and in Remark \ref{Operator D contraction}(ii), respectively. Since the associated $*$-representation $\rho: C(\overline{\mathbb D}) \to B(\mathfrak{H}_1)$ is a simple representation, following the notation of the preceding discussion, we get $$ \widetilde{\pi}(X)=\begin{bmatrix}
     \rho(E(T)) & 0\\ 0 & I_{\mathfrak{H}_2} \otimes T_2 
 \end{bmatrix}, \mbox{~and~} \widetilde{\pi}(DX)=(D\otimes I_{\mathfrak{K}})\widetilde{\pi}(X).$$
 A routine verification shows that
\beq \label{action of pi}
\widetilde{\pi}((DX)^n)=(\widetilde{\pi} (DX))^n \notag \mbox{ and } \\  \widetilde{\pi}(X((DX)^n))=\widetilde{\pi}(X) \widetilde{\pi}((D X)^n)\; \text{for all } n \in \mathbb N.
\eeq 
 For ease of notation, denote $\psi_j=\psi_{\alpha_j},$ $j=1, \ldots,  N.$ 
  For $h \in \mathfrak{H}_1$, $g\in \mathfrak{K}$, we have
\begin{align*}
    \left\|\widetilde{\pi} (\rho(E)) (h\otimes g)\right\|^2
&= \sum_{j=1}^N \inp{P_j h}{h} \inp{\psi_j(T)^* \psi_j (T)g}{g}\\
&< \sum_{j=1}^N \inp{P_j h}{h} \|g\|^2, \text{ as } \|\psi_j(T)\| <1\\
&=\|h\|^2\|g\|^2.\\
\end{align*}

Therefore, $\|\widetilde{\pi}(\rho(E))\| <1.$  Also, since $\|T_2\| <1$, 
therefore $ \label{pifmaction} \|\widetilde{\pi}(X)\|<1.$

Define a sequence $\{f_m\}$ of holomorphic functions by \beqn 
 f_m(z) = A +B X(z)\left (\sum_{n=1}^m (DX(z))^n\right) C \mbox{ for all } m\in \mathbb N.\eeqn 
Since $\|D X(z)\| <1$, the sequence $\{f_m\}$ converges to $f$ pointwise on $\mathbb E$.
  Now by \eqref{action of pi}, we get
  \beqn 
  \pi(f_m)=  A\otimes I_{\mathfrak{K}} +(B\otimes I_{\mathfrak{K}}) \widetilde{\pi} (X) \left(\sum_{n=1}^m (\widetilde{\pi}(DX))^n\right) (C\otimes I_{\mathfrak{K}}).
  \eeqn 
  
  Since $\|\widetilde{\pi}(X)\| <1, \; \|\widetilde{\pi}(DX)\|=\|(D\otimes I_{\mathfrak{K}})\widetilde{\pi}(X)\| <1.$ Therefore, $\{\pi(f_m)\}$ converges in the weak operator topology to $$A\otimes I_{\mathfrak{K}} +(B\otimes I_{\mathfrak{K}}) \widetilde{\pi} (X) ((I_{\mathfrak{H}}\otimes I_{\mathfrak{K}})-\widetilde{\pi}(DX))^{-1} (C\otimes I_{\mathfrak{K}})=\pi(f).$$ 
The preceding equality follows directly from Lemma \ref{properties of SA}(iv).

Consider $A'=A \otimes I_{\mathfrak{K}}, B'=B\otimes I_{\mathfrak{K}}, C'=C\otimes I_{\mathfrak{K}}$, and $D'=D \otimes I_{\mathfrak{K}}.$ 
Then
\beqn I_{\mathfrak{K}}-\pi(f)^* \pi(f)= {C'}^* {(I_{\mathfrak{K}}-D' \widetilde{\pi}(X))^{-1}}^* (I_{\mathfrak{K}}-\widetilde{\pi}(X)^* \widetilde{\pi}(X)) (I_{\mathfrak{K}}-D' \widetilde{\pi}(X))^{-1} C'.\eeqn
Since $\|\widetilde{\pi}(X)\| <1$, this yields $I_{\mathfrak{K}}-f(T)^* f(T) \geq 0.$ This completes the proof.
 \end{proof}

 \begin{lemma}\label{hat}
        Let $f \in A(\overline{\mathbb E}) \cap \SAE.$ Define $\hat{f}:\overline{\mathbb E}\rightarrow\mathbb C$ by $\hat{f}(z)=\overline{f(\bar z)}.$ Then $\hat{f} \in A(\overline{\mathbb E}) \cap \SAE.$
    \end{lemma}
     \begin{proof}
        Let $f \in A(\overline{\mathbb E}) \cap \SAE.$ Since $\mathbb E$ is a polynomially convex domain, by Oka-Weil theorem, there exists a sequence $ \{p_k\}$ of polynomials converging to $f$ uniformly on $\overline{\mathbb E}$. Then the sequence $\{\hat{p}_k\}$ given by $\hat{p}_k(z)=\overline {p(\bar z)},$ converges to $\hat f$ uniformly. It follows that $\hat{f} \in A(\overline{\mathbb E}).$ Let $(T,\mathfrak{K}) \in \mathcal{M}.$ Arguing as in the proof of part $(iv)$ of Lemma \ref{properties of SA}, $\hat{p}_k(T)$ and $p_k(T)$  converges to $\hat{f}(T)$ and $f(T),$ respectively, in the weak operator topology. By Lemma \ref{spectrum}, $\sigma(T), \text{and } \sigma(T^*)$ are subsets of $\mathbb E.$ For any $x,y\in \mathfrak{K}$, 
        $$ |\inp{\hat{f}(T) x}{y}| =\lim_{k\to \infty} |\inp{\hat{p}_k(T)x}{y}|=\lim_{k\to \infty}|\inp{x}{p_k(T^*)y}|=|\inp{x}{f(T^*)y}|.$$
      Consequently, $\|\hat{f}(T)\| \leq 1$ for any $T\in \mathcal{M},$ that is, $\hat{f}\in \SAE.$
    \end{proof}
\section{Proof of the realization theorem} \label{S3}
In this section, we record the proof of Theorem \ref{realization theorem}.
\begin{proof}[Proof of Theorem \ref{realization theorem}:]
   
 $\mathrm{(i)} \implies \mathrm{(ii)}.$ We first establish the result for $f \in \SAE \cup A(\overline{\mathbb E}),$ and then use the radial technique to complete the proof.
 Let $k \in \AKE$ and $M_z=(M_{z^{(1)}}, M_{z^{(2)}}, M_{z^{(3)}})$ be the tuple of operators of multiplication by the coordinate functions $z^{(i)},\; i=1,2,3,$ on the reproducing kernel Hilbert space $H(k)$ determined by $k.$ Then each $M_{z^{(i)}}$ is bounded, in fact, $\|M_{z^{(2)}}\| \leq 1$ and $\|\psi_\alpha(M_z)\|\leq 1$ for all $\alpha \in \overline{\mathbb D}.$

For $n \in \mathbb N$ and $z_1, \ldots, z_n \in \mathbb E,$ let $H_n$ denote the finite dimensional Hilbert space spanned by $\{k(\cdot, z_1), \ldots, k(\cdot, z_n)\}$.  Consider the tuple $S=(S
_1,S_2,S_3)$ defined on $H_n$ as follows: 
    \beqn
S_j^*k( \cdot, z_l)=\bar {z}_l^{(j)}k( \cdot, z_l) \quad l=1, \ldots n,~ j=1,2,3.
    \eeqn
It is easy to see that $H_n$ is jointly invariant under $M_z^*$ and $M^*_{z^{(j)}}|_{H_n}=S_j^*$ for all $j=1,2,3$. Furthermore,  $\|S_2^*\| \leq 1$ and $\|\psi_\alpha(S^*)\| \leq 1$ for all $\alpha \in \overline{\mathbb D}.$ 
For any $0<t<1,$ $\alpha \in \overline{\mathbb D}$ and for any scalars $c_1, \ldots, c_n,$ we have 
\beq \label{psi sdott}
\|(t^2\alpha S_3^*-tS_1^*)(\sum_{j=1}^n c_jk(\cdot, z_j))\|
\notag&\leq & t\|(t\alpha  M_{z^{(2)}}^*-I)(\sum_{j=1}^n c_jk(\cdot, z_j))\|\\
&=& t\|(t\alpha  S_2^*-I)(\sum_{j=1}^n c_jk(\cdot, z_j))\|.
\eeq
The inequality above follows from the fact that $\|\psi_{t\alpha}(M_z^*)\| \leq 1.$
Since $\|t  S_2^*\| <1,$ it follows from \eqref{psi sdott} that $t \cdot S^*=(tS_1^*, tS_2^*, t^2 S_3^*) \in \mathcal{M}$ for any $0 <t <1.$ 

Let $f\in \SAE \cap A(\overline{\mathbb E})$ and $\hat{f}$ be the function defined as $\hat{f}(z)=\overline{f(\bar{z})},$ then by Lemma \ref{hat}, $\hat{f}\in \SAE \cap A(\overline{\mathbb E}).$  Thus, there exists a sequence of polynomials $\{p_m\}$ such that $\hat{p}_m$ converges to $\hat{f}$ uniformly on $\overline {\mathbb E}$. Therefore,
\beqn \left \|\hat{f}(t \cdot S^*)\left(\sum_{j=1}^n c_j k(\cdot,z_j)\right)\right\|&=& \lim_{m \to \infty} \left \|\hat{p}_m(t \cdot S^*)\left(\sum_{j=1}^n c_j k(\cdot,z_j)\right)\right \|\\
&=&\lim_{m \to \infty} \left \|\left(\sum_{j=1}^n c_j\overline{p_m(t \cdot z_j)} k(\cdot,z_j)\right)\right \|\\
&=&\left \|\left(\sum_{j=1}^n c_j \overline{f(t \cdot z_j)}k(\cdot,z_j)\right)\right\|.\eeqn
As $\hat{f}\in \SAE$, we have 

   $$ \left \|\left(\sum_{j=1}^n c_j \overline{f(t \cdot z_j)}k(\cdot,z_j)\right)\right\|^2 \leq \left 
   \|\sum_{j=1}^n c_j k(\cdot,z_j)\right \|^2.$$
   This shows that $\sum_{i,j=1}^n \bar{c}_i c_j (1-f(t \cdot z_i)\overline{f(t \cdot z_j)})k(z_i,z_j) \geq 0.$ By taking $t \to 1^-,$ we have $\rm{(ii)}$ for $f \in \SAE \cap A(\overline{\mathbb E}).$ 
Let $f\in \SAE$ and $0 <r<1$, consider the function $f_r:\overline{\mathbb E} \to \mathbb C$ given by $f_r(z)=f(r \cdot z)$. Then 
$f_r \in \SAE \cap A(\overline{\mathbb E})$. Consequently, $$\sum_{i,j=1}^n \bar{c}_i c_j(1-f_r(z_i) \overline{f_r(z_j)})k(z_i,z_j) \geq 0. $$
We conclude the proof by taking $r\to 1^-.$

$\mathrm{(ii)} \implies \mathrm{(iii)}.$ Consider the function $g$ on $\mathbb E \times \mathbb E$ given by $g(z,w)=1-f(z)\overline{f(w)}$ for all $z,w \in \mathbb E.$ Since $g$ is self-adjoint on $\mathbb E,$ an application of Theorem \ref{keytheorem} yields this implication.

$\mathrm{(iii)} \implies \mathrm{(iv)}.$  We use the idea of lurking isometry to prove this part (see \cite[Theorem 11.13]{AM2002}). This idea has also been used to prove similar types of results, such as \cite[Theorem 2.2]{DM2007}, \cite[Theorem 3.1]{AY2017}, \cite[Theorem 1.5]{BB2004}, and \cite[Theorem 5.2]{BS2022}. By hypothesis, there exist $\Gamma \in \CDE$ and $\Delta \in \PKE$ such that 
\beq \label{keyequation2} 1 - f(z) \overline{f (w)}=\Gamma(z, w)\Big(1- E(z)\overline{E(w)}\Big)+ (1-z^{(2)}\overline{w}^{(2)})\Delta(z,w), \quad z,w\in \mathbb E.
\eeq

Since $\Gamma \in \CDE,$  by Proposition \ref{decompositionresult}, there exist a Hilbert space $\mathfrak{H}_1$ and a function $L:\mathbb E \to \mathcal{B}(C(\overline{\mathbb D}),\mathfrak{H}_1)$ such that $\Gamma(z,w)(h_1\bar{h}_2)=\inp{L(z)(h_1)}{L(w)(h_2)}_{\mathfrak{H}_1}$ for all $z,w\in \mathbb E$ and $h_1,h_2 \in C(\overline{\mathbb D}).$ Moreover, there exists a unital $*$-representation $\rho: C(\overline{\mathbb D}) \to \mathcal{B}(\mathfrak{H}_1)$ such that $\rho(h_1)(L(z)h_2)=L(z)(h_1h_2).$

  On the other hand, as $\Delta \in \PKE$, by \cite[Theorem 2.53]{AM2002}, there exist a Hilbert space $\mathfrak{H}_2$ and a function $g:\mathbb E \to \mathfrak{H}_2$ such that $\Delta(z,w)=\inp{g(z)}{g(w)}_{\mathfrak{H}_2}.$ Then \eqref{keyequation2} becomes 
  \beqn
    1-f(z)\overline{f(w)}
    &=& \inp{L(z)1}{L(w)1}_{\mathfrak{H}_1}-\inp{L(z)(E(z))}{L(w)(E(w))}_{\mathfrak{H}_1}\\&+& \inp{g(z)}{g(w)}_{\mathfrak{H}_2}-\inp{z^{(2)}g(z)}{w^{(2)}g(w)}_{\mathfrak{H}_2}.
  \eeqn
  After rearranging terms, we get
  \beqn 1&+& \inp{\rho(E(z))(L(z)1)}{\rho(E(w))(L(w)1)}_{\mathfrak{H}_1}+ \inp{z^{(2)}g(z)}{w^{(2)}g(w)}_{\mathfrak{H}_2}\notag \\&=& f(z)\overline{f(w)} +\inp{L(z)1}{L(w)1}_{\mathfrak{H}_1}+\inp{g(z)}{g(w)}_{\mathfrak{H}_2}.\eeqn
  Equivalently, \beq \label{isometryeq} 1+\left \langle X(z)\left(\substack{L(z) 1\\ g(z)}\right), X(w)\left(\substack{L(w) 1\\ g(w)}\right) \right \rangle_{\mathfrak{H}}=f(z)\overline{f(w)}+\left \langle \left(\substack{L(z) 1\\ g(z)}\right), \left(\substack{L(w) 1\\ g(w)}\right) \right \rangle_{\mathfrak{H}},\eeq where $$X(z)= \begin{bmatrix}
    \rho(E(z))& 0\\ 0 & z^{(2)} I_{\mathfrak{H}_2}
 \end{bmatrix} \in \mathcal{B}(\mathfrak{H}),\; \mathfrak{H}=\mathfrak{H}_1 \oplus \mathfrak{H}_2.$$
 Consider the subspaces $\mathfrak{H}_d$ and $\mathfrak{H}_r$ of $\mathbb C\oplus \mathfrak{H}$ given by
 $$\mathfrak{H}_d=\operatorname{span} \left\{\left(\substack{1\\ \;\;\\ X(z)\left(\substack{L(z)1\\ g(z)}\right)}\right): z\in \mathbb E\right\},\;\; \mathfrak{H}_r=\operatorname{span}  \left\{\left( \substack{f(z)\\ \;\;\\\left(\substack{L(z)1\\ g(z)}\right) }\right): z\in \mathbb E\right\}.$$ Define a linear map $V:\mathfrak{H}_d \to \mathfrak{H}_r$ by $$V\left(\substack{1\\ \;\;\\ X(z)\left(\substack{L(z)1\\ g(z)}\right)}\right)= \left( \substack{f(z)\\ \;\;\\\left(\substack{L(z)1\\ g(z)}\right) }\right).$$
 By \eqref{isometryeq}, $V$ can be extended isometrically on $\overline{\mathfrak{H}_d}$. Now, adding an infinite-dimensional summand to $\mathfrak{H}$ if necessary, one can extend $V$ to an isometry from $\mathbb C \oplus \mathfrak{H} \to \mathbb C \oplus \mathfrak{H}$ (see \cite[Section 11.3]{AM2002}). 
 
 Let $\begin{bmatrix}
     A & B\\ C & D
 \end{bmatrix}$ be the decomposition of $V$ with respect to $\mathbb C \oplus \mathfrak{H}$. Then
\beq
    &\label{1st} A1 +B X(z)\left(\substack{L(z)1\\ g(z)}\right)= f(z), \\
    & \label{2nd} C1 + D X(z)\left(\substack{L(z)1\\ g(z)}\right)=\left(\substack{L(z)1\\ g(z)}\right).
\eeq
From \eqref{2nd} and Remark \ref{Operator D contraction}, we have $$\left(\substack{L(z)1\\ g(z)}\right)= (I_\mathfrak{H}-DX(z))^{-1}C1.$$ This combined with \eqref{1st} yields that $$f(z)=A+ B X(z)(I_\mathfrak{H}-DX(z))^{-1}C.$$
This completes the proof of $\mathrm{(iii)} \implies \mathrm{(iv)}.$

$\mathrm{(iv)} \implies \mathrm{(i)}:$ To prove this implication, we follow the idea of the proof of \cite[Proposition 3.2]{DM2007}.
Let $f \in \UCE.$ First we construct a net of functions $ f_\beta $ in $ \UCE$, each associated with a simple representation, such that $ f_\beta $ converges pointwise to $ f $ and then we complete the proof using Lemma~\ref{simplerep} and the weak continuity of $ \pi $.

 Consider a directed set $\mathcal{F}$ consisting of pairs $(F,\epsilon)$, where $F$ is a finite subset of $\mathbb E$ and $\epsilon >0,$ ordered by $(F,\epsilon) \leq (G,\delta)$ if $F\subseteq G$ and $ \epsilon \geq \delta.$ Let $\beta =(F,\epsilon) \in \mathcal{F}.$ By the compactness of $\{\psi_\alpha: \alpha \in \overline{\mathbb D}\}$, there exists a finite collection $\{\Delta_1^\beta, \ldots, \Delta_m^\beta\}$ of mutually disjoint nonempty sets covering $\{\psi_\alpha: \alpha \in \overline{\mathbb D}\}$ with the property that for each $j=1, \ldots, m,$
 \beq \label{disjointcollection}
|\psi'(z)-\psi''(z)|<\epsilon, \quad \psi',\psi'' \in \Delta_j^\beta,\; z\in F.
 \eeq
 Also, consider a collection $\{\psi_1^\beta,\ldots,\psi_m^\beta\}$ such that $\psi_i^\beta=\psi_{\alpha_i}$ for some $\alpha_i \in \overline{\mathbb D}$ and $\psi_i^\beta \in \Delta_i^\beta$ for $i=1, \ldots, m.$

Since $f\in \UCE,$ there is an asoociated unital $*$-representation  $\rho: C(\overline{\mathbb D}) \to \mathcal{B}(\mathfrak{H}_1)$ (see Definition \ref{unitarycolligation}).
Therefore, by \cite[Theorem 9.8]{CJ2000}, there exists a unique $\mathcal{B}(\mathfrak{H}_1)$-valued (spectral) measure $\mu$ on the Borel subsets of $\overline{\mathbb D}$  such that 
 $$\rho(h)=\int_{\overline{\mathbb D}} h(\alpha) d\mu(\alpha) \text{ for all } h \in C(\overline{\mathbb D}).$$
Define $\rho_{_\beta}:C(\overline{\mathbb D})\to \mathcal{B}(\mathfrak{H}_1)$ by 
$$\rho_{_\beta}(h)=\sum_{j=1}^m \mu(\Delta_j^\beta) h(\alpha_j) \text{ for all } h \in C(\overline{\mathbb D}).$$ 
Since $ \Delta_1^\beta, \ldots, \Delta_m^\beta $ is a partition of $\{\psi_\alpha: \alpha \in \overline{\mathbb D}\} \cong \overline{\mathbb D}$, it follows that the operators $ \mu(\Delta_j^\beta) $ are pairwise orthogonal projections for $ j = 1, \ldots, m $, and satisfy $ \sum_{j=1}^m \mu(\Delta_j^\beta) = I_{\mathcal{H}_1} $. Consequently, $\rho_{_\beta}$ is a simple representation. 
For any $z\in F$, we have \beqn \left\|\rho_{_\beta} (E(z))-\rho(E(z))\right\| &=& \left\| \sum_{j=1}^m \psi_{j}^\beta(z) \mu(\Delta_j^\beta)-\int_{\overline{\mathbb D}} E(z)(\alpha)d\mu(\alpha)\right\|\\
 &\leq & \sum_{j=1}^m \left\|\int_{\Delta_j^\beta} (\psi_j^\beta(z)-\psi_\alpha(z))d\mu(\alpha)\right\| \overset{\eqref{disjointcollection}}{<}\epsilon.\eeqn
This shows that for every $z\in \mathbb E$ there is a $\beta =(F,\epsilon)\in \mathcal{F}$ with $z\in F$ such that $\|\rho_{_\beta}(E(w))-\rho(E(w))\| < \epsilon$ for all $w\in F.$ 
 Moreover, for any $\gamma \in \mathcal{F}$ with $\beta \leq \gamma=(G,\delta)$, we have $\| \rho_\gamma(E(w))-\rho(E(w))\| < \epsilon$ for all $w \in G.$
 
 For each $\beta \in \mathcal{F},$ define $$f_\beta(z)=A+B X_\beta(z)(I_\mathfrak{H}-DX_\beta(z))^{-1}C \text{ for all } z \in \mathbb E,$$ where $$X_\beta (z)=\begin{bmatrix}
     \rho_{\beta}(E(z))& 0\\ 0 & z^{(2)} I_{\mathfrak{H}_2}
 \end{bmatrix}.$$  Note that $f_\beta \in \UCE$ with associated simple representation $\rho_\beta.$ 
 Thus, by Lemma \ref{simplerep}, $\|\pi(f_\beta)\| \leq 1$ for all $\beta\in \mathcal{F}$.
 For the rest of the proof, we fix $z\in \mathbb E$ and $\epsilon' >0.$ Consider $\beta'=(F,\epsilon')\in \mathcal{F}$ such that $z\in F$. As $\|E(w)\|_{\infty} <1,$ let $\delta=\min\{\frac{1-\|E(w)\|_\infty}{2}: w\in F\}$ and $r_1=1-\delta/2.$ Since $ 2\delta \leq 1-\|E(w)\|_{\infty}$ and $\rho$  is a unital $*$-representation, $\|\rho(E(w))\| \le 1-2\delta$ for any $w\in F.$ Let $ \beta = (F, \epsilon) \in \mathcal{F}$ with $0 < \epsilon < \min\{\delta/2, 
 \epsilon'\}$. It follows that $ \beta' \leq \beta $ and for any $ w \in F $, the inequality $ \| \rho_\beta(E(w)) - \rho(E(w)) \| < \epsilon $ yields that $ \| \rho_\beta(E(w)) \| < r_1 $.
 Furthermore, \beq \label{XbetamiXlessthanepsi} \|X_\beta(w)-X(w)\|\leq \epsilon\; \text{for any } w \in F.\eeq
Choose $r_2$ such that  $|w^{(2)}| <r_2 <1$ for all $w \in F$ and $r=\max \{r_1,r_2\}.$
Therefore, 
$ \|X_\beta(w)\| <r$ and subsequently,
\beq \label{DXbetalessthanepsi}
\|D X_\beta(w)\| <r \mbox{ for all } w\in F.
\eeq
Since $\|\rho(E(w))\| \leq 1-2\delta < r_1,\; \|X(w)\| <r,$ and therefore
\beq \label{DXlessthanepsi}
\|DX(w)\| <r \mbox{ for all }w\in F.
\eeq
Now for any $n\in \mathbb N$,
\begin{align*}
    (DX_\beta)^n-(DX)^n
&= D(X_\beta-X)(DX_\beta)^{n-1} + (DX) D(X_\beta-X)(DX_\beta)^{n-2}\\
&+\cdots+ (DX)^{n-1} D (X_\beta-X).
\end{align*}
This together with \eqref{XbetamiXlessthanepsi}--\eqref{DXlessthanepsi} gives that \beqn
\| (D X_\beta(w))^n-(DX(w))^n\| \leq n\|D\| \epsilon r^{n-1} \mbox{ for all } w \in F \text{ and } n \in \mathbb N.
\eeqn
Thus, for any $w\in F,$ we get 
 \begin{align*}
&\| X_\beta(w)(I_\mathfrak{H}-D X_\beta(w))^{-1}-X(w) (I_\mathfrak{H}-D X(w))^{-1}\| \\
&\leq \| (X_\beta(w)-X(w))\| \|(I_\mathfrak{H}-DX_\beta(w))^{-1}\|\\
 & \;\;\;\;+ \|X(w)\|\|[(I_\mathfrak{H}-D X_\beta(w))^{-1}-(I_\mathfrak{H}-D X(w))^{-1}]\|\\
 &\leq \epsilon \left(\frac{1}{1-r}+\frac{r}{(1-r)^2}\right).
 \end{align*}
Let $\gamma\in \mathcal{F}$ with $\beta \leq \gamma,$ then
\beqn
  |f_\gamma(z)-f(z)| &=& |B X_\gamma(z)(I_\mathfrak{H}-D X_\gamma(z))^{-1}C-BX(z) (I_\mathfrak{H}-D X(z))^{-1} C|\\  
  &\leq & \|B\| \|C\| \epsilon \left[\frac{1}{1-r}+\frac{r}{1-r^2}\right] < \|B\| \|C\| \epsilon' \left[\frac{1}{1-r}+\frac{r}{1-r^2}\right].
\eeqn
Therefore, $f_\gamma$ converges to $f$ pointwise on $\mathbb E$. Since $f_\gamma$ is uniformly bounded, by \cite[Theorem 1.4.31]{VS2005}, there exists a subsequence $ \{f_{\gamma_k}\} $ converging uniformly to $ f $. Since each $ \pi(f_{\gamma_k}) $ is a contraction, by Lemma~\ref{properties of SA}(iv), $ \pi(f) $ is a contraction. This completes the proof.
 \end{proof}

\section{Interpolation and Extension Theorem} \label{S4}
In this section, we discuss an analog of the interpolation and extension theorem for tetrablock.
\subsection{Interpolation} \label{ss4.1} An application of the realization theorem is the characterization of interpolation problems via some positivity condition(s). For instance, the positivity of the Pick matrix resolves the Nevanlinna--Pick interpolation problem on the unit disc $\mathbb{D}$ (see \eqref{basic pick matrix}).
In higher-dimensional domains, the single positivity condition may not be sufficient. Instead, it is characterized by the positive semi-definiteness of the matrices
$$\begin{bmatrix}(1 - w_i \overline{w}_j) k(\lambda_i, \lambda_j)\end{bmatrix}_{i,j=1}^n,$$
for every kernel $k$ belonging to a family of kernels that is ``closely" associated with the domain and typically contains the Szeg\"{o} kernel of the domain (for instance, see \cite[Theorem 11.49]{AM2002}, \cite[p.~508]{BS2018}).
Below, we prove an analogous result (Theorem \ref{interpolation theorem}), where the interpolating functions are in $\SAE.$

\begin{proof}[The proof of Theorem \ref{interpolation theorem}:]

 $\mathrm{(i)} \implies \mathrm{(ii)}:$ Let $f\in \mathcal{SA}(\mathbb E)$ interpolates $z_1,\ldots, z_n \in \mathbb E$ to $w_1,\ldots,w_n \in \overline{\mathbb D}.$ 
 Equivalence of $(i)$ and $(ii)$ of Theorem \ref{realization theorem} proves this implication.

          $\mathrm{(ii)} \implies \mathrm{(iii)}:$
          Let $F=\{z_1,\ldots,z_n\}.$ Consider a map $g:F\times F \to \mathbb C$ defined by $g(z_i,z_j)=1-w_i\overline{w}_j$ for all $1 \Le i,j \Le n.$ Since $g$ is self-adjoint and $gk\succcurlyeq 0$ on $F$ for each $k \in\AKE,$ following the proof of Theorem \ref{keytheorem}, we get the desired result. 
          
     $\mathrm{(iii)} \implies \mathrm{(i)}:$ 
      Assume that $\mathrm{(iii)}$ holds. Then by Proposition \ref{decompositionresult} and \cite[Theorem 2.53]{AM2002}, there exist Hilbert spaces $\mathfrak{H}_1$, $\mathfrak{H}_2$, and maps $L: F \to \mathcal{B}(C(\overline{\mathbb D}), \mathfrak{H}_1)$, $q: F \to \mathcal{B}(\mathfrak{H}_2)$ and a unital $*$-representation $\rho: C(\overline{\mathbb D}) \to \mathcal{B}(\mathfrak{H}_1)$ such that the equation \eqref{interpolationeq1} becomes
          \begin{align*}
        &1+ \inp{\rho(E(z_i))L(z_i)1}{\rho(E(z_j))L(z_j)1}_{\mathfrak{H}_1}+ \inp{z_i^{(2)}q(z_i)}{z_j^{(2)}q(z_j)}_{\mathfrak{H}_2}\\
       & = w_i\overline{w}_j +\inp{L(z_i)1}{L(z_j)1}_{\mathfrak{H}_1}+\inp{q(z_i)}{q(z_j)}_{\mathfrak{H}_2}.
    \end{align*}  
             
    Let $\mathfrak{H}=\mathfrak{H}_1 \oplus\mathfrak{H}_2.$ Consider the following (finite dimensional) subspaces of $\mathbb C\oplus \mathfrak{H}$ \beqn \mathfrak{H}_d=\rm{span} \left\{\left(\substack{1\\\;\\ \rho(E(z_i))L(z_i)1\\\;\\ z_i^{(2)} q(z_i)}\right): 1 \le i\le n\right\},\quad \mathfrak{H}_r=\rm{span} \left\{\left(\substack{w_i\\\;\\ L(z_i)1\\\;\\  q(z_i)}\right): 1 \le i\le n\right\}.\eeqn 
    
    Define $V:\mathfrak{H}_d \to \mathfrak{H}_r$ by \beq\label{iso-int} V \left(\substack{1\\\;\\ \rho(E(z_i))L(z_i)1\\\;\\ z_i^{(2)} q(z_i)}\right)=\left(\substack{w_i\\\;\\ L(z_i)1\\\;\\  q(z_i)}\right)\; \text{for all }i=1,\ldots,n.\eeq 
 We now extend $V$ to an unitary operator on $\mathbb C \oplus \mathfrak{H}$ and write $V=\begin{bmatrix}
        A & B\\ C & D
    \end{bmatrix}.$

By equivalence of $(iv)$ and $(i)$ of Theorem \ref{realization theorem}, the function $f:\mathbb E\to \mathbb C$ defined by $$f(z)=A +B X(z) \left(I_{\mathfrak{H}}-DX(z)\right)^{-1}C,\; \text{where } X(z)=\begin{bmatrix}
        \rho(E(z)) &0\\ 0 & z^{(2)}I_{\mathfrak{H}_2}
    \end{bmatrix}$$ is in the class $\SAE.$ By \eqref{iso-int}, it is easy to verify that $f(z_i)=w_i$ for all $i=1,\ldots,n$. This completes the proof.
\end{proof}

\subsection{Extension theorem} \label{S6}

Recall that $\mathcal{M}'$ is the collection of all commuting $3$-tuples $T=(T_1, T_2, T_3)$ with $\sigma(T)\subseteq \mathbb E$ such that  $$\|T_2\| \leq 1 \mbox{ and } \|\psi_{\alpha}(T)\|\leq1 \text{ for all } \alpha \in \overline{\mathbb D}.$$ 

Clearly, $\mathcal{M}\subseteq \mathcal{M}'$ and for any $0 <r<1$ and $T\in \mathcal{M}', \; r \cdot T=(rT_1,rT_2,r^2T_3)\in \mathcal{M}$.

We first prove the following result, which establishes the necessary condition of the Theorem \ref{extension theorem}.

\begin{lemma} \label{motivatingextension}
    Let $f \in \mathcal H \mathcal E(V)$. If there exists $g \in H^{\infty}(\mathbb E)$ such that $g|_{V}=f,$ $\frac{1}{\|f\|_{\infty}}g\in \SAE$  and $\|g\|_\infty=\|f\|_{\infty},$ then \beqn \|f(S)\| \leq \|f\|_{\infty}\text{ for all }S \in \mathcal M'\text{ subordinate to }V.\eeqn
\end{lemma}

\begin{proof}
   Let $(S, \mathfrak{K)} \in \mathcal M'$ be subordinate to $V.$ 
   Let $\{r_n\}$ be a sequence such that $0 <r_n <1$ and $r_n \to 1.$ Define $g_{n}:\mathbb E \to \mathbb C$ by $g_{n}(z)=g(r_n \cdot z)$ for all $n\in \mathbb N.$
Then $\{g_{n}\}$ is a sequence of holomorphic functions that converges to $g$ pointwise. Since $g$ is bounded, $\{g_{n}\}$ is uniformly bounded. By using the dominated convergence theorem, we get that $g_{n}(S)$ converges to $g(S)$ in the weak operator topology. 
Now for $x,y \in \mathfrak{K}$, \beqn
|\inp{g(S)x}{y}|=\lim_{n \to \infty} |\inp{g(r_n \cdot S) x}{y}|  \leq \lim_{n \to \infty} \|g(r_n \cdot S)\| \|x\| \|y\| \leq \|f\|_{\infty} \|x\| \|y\|,\eeqn
where the last inequality follows from the fact that $r_n \cdot S \in \mathcal M$ for each $n \in \mathbb N$. Therefore, $\|f(S)\|=\|g(S)\| \leq \|f\|_{\infty}.$ 
\end{proof}

Now for the proof of the sufficient part of Theorem \ref{extension theorem}, we follow a familiar technique through the interpolation theorem, as used in the proof of  \cite[Theorem 1.13]{AMc2003} for bidisc, \cite[page 506]{BS2018} for symmetized bidisc.   
 
 For any $g\in H^\infty(\mathbb E)$, let $\|g\|_{\mathcal{M}}=\mathrm{sup}\{\|g(T)\|: T\in \mathcal{M}\}$. For $f\in \mathcal{H}\mathcal{E}(V),$ set $$\mathcal{SA}_f(\mathbb E)=\{g\in H^\infty(\mathbb E): \|g(T)\| \leq \|f\|_{\infty}\; \text{for all } T\in \mathcal{M}\}.$$ Note that $\SAFE$ is nonempty. Further, for $g \in \SAFE,$ $\|g\|_{\mathcal{M}} \leq \|f\|_{\infty}.$ Now for the given data $Z=\{z_1,\ldots,z_n\}\subseteq V$ and $W=\{w_1,\ldots,w_n\}\subseteq \mathbb C$, consider $$\rho^f(Z,W)=\mathrm{inf}\{\|g\|_\mathcal{M}: g\in \mathcal{SA}_f(\mathbb E)\; \text{and } g(z_i)=w_i,\;i=1,\ldots, n\}.$$
The following lemma shows that $\rho^f(Z,W)$ is attained.
\begin{lemma} \label{extremalfunction}
    Let $f\in \mathcal{H}\mathcal{E}(V).$ For $Z=\{z_1,\ldots,z_n\}\subseteq V$ and $W=\{w_1,\ldots,w_n\}\subseteq \mathbb C,$ there exists $g \in \SAFE$ such that $g(z_i)=w_i$ for all $i=1,\ldots,n$ and $\|g\|_\mathcal{M}=\rho^f(Z,W).$
\end{lemma}
\begin{proof}
Let $\{g_m\}$ be a sequence in $\mathcal{SA}_f(\mathbb E)$ such that $g_m(z_i)=w_i$ for all $i=1,\ldots,n$ and 
$$\rho^f(Z,W)=\lim_{m\to \infty} \|g_m\|_\mathcal{M}.$$ 
Note that for any $g\in \mathcal{SA}_f(\mathbb E)$, $\|g\|_\infty \leq \|g\|_{\mathcal{M}}\leq \|f\|_{\infty}.$ Consequently, there is a subsequence $\{g_{m_k}\}$ of $\{g_m\}$ that converges to g in $H^\infty(\mathbb E).$ Applying a similar argument involving the dominated convergence theorem, as used in the proof of Lemma \ref{motivatingextension}, we obtain
$\|g(T)\| \leq \|f\|_{\infty}$ for all $T \in \mathcal{M}.$
 Thus, $g\in \mathcal{SA}_f(\mathbb E)$ and $g(z_i)=w_i$ for all $i=1,\ldots,n.$ This yields that $\rho^f(Z,W) \leq \|g\|_\mathcal{M}.$ Now it remains to show that $\|g\|_\mathcal{M}\leq \rho^f(Z,W).$
Let $(T, \mathfrak{K}) \in \mathcal{M}.$ Then for $x,y\in \mathfrak{K}$, we have $$|\inp{g(T)x}{y}|=\lim_{k\to \infty}|\inp{g_{m_k}(T)x}{y}|\leq \rho^f(Z,W) \|x\| \|y\|.$$ Thus $\|g(T)\|\leq \rho^f(Z,W)$ for each $T \in \mathcal{M}.$ This completes the proof.  
\end{proof}

For a given  $n$ distinct points $z_1,\ldots,z_n \in \mathbb E,$ let $\mathcal{K}_z$ denote the set of all $n\times n$ strictly positive definite matrices $k=\begin{bmatrix}
    k(i,j)
\end{bmatrix}_{i,j=1}^n$ satisfying $k(i,i)=1$ for $i=1,\ldots,n,$ \beq \label{condition 1}\begin{bmatrix} (1-z_i^{(2)}\bar{z}_j^{(2)})k(i,j) \end{bmatrix}_{i,j=1}^n \geq 0,\; \text{and } \begin{bmatrix} (1-\psi_\alpha(z_i)\overline{\psi_\alpha(z_j)})k(i,j)\; \end{bmatrix}_{i,j=1}^n \geq 0.\eeq
For the ease of reference, we restate the equivalence of (i) and (ii) from Theorem \ref{interpolation theorem} in the context of $\mathcal{K}_z.$ 

\begin{theorem}\label{interpolation rephrased}
    Let $w_1,\ldots,w_n \in \overline{\mathbb D}.$ Then the $n\times n$ matrix $\begin{bmatrix} (1-w_i \bar{w}_j)k(i,j)\end{bmatrix}_{i,j=1}^n$ is positive semi-definite for all $k\in \mathcal{K}_z$ if and only if there is a function $g\in \SAE$ with the property $g(z_i)=w_i$ for $i=1,\ldots,n.$
\end{theorem}
In view of Lemma \ref{extremalfunction},  we call a function  $g\in \mathcal{SA}_f(\mathbb E)$ \textit{extremal} if it satisfies $\|g\|_\mathcal{M}=\rho^f(Z,W),$ and $g(z_i)=w_i,~ i=1, \ldots, n,$ for the given data $Z=\{z_1,\ldots,z_n\} \subset \mathbb E$ and $W=\{w_1,\ldots,w_n\}\subseteq \mathbb C.$

\begin{lemma}\label{lemma for extension}
    Let $f\in \mathcal{H}\mathcal{E}(V).$ If $g$ is an extremal function for the data $Z=\{z_1,\ldots,z_n\}\subseteq V,\; W=\{w_1,\ldots,w_n\}\subseteq \mathbb C$, then there is a commuting $3$-tuple $S=(A,B,P)\in \mathcal{M}'$ subordinate to $Z$ such that $\|g(S)\|=\rho^f(Z,W).$
\end{lemma}
\begin{proof}
    We first show that $\mathcal{K}_z$ is compact. Since the trace of each $k \in \mathcal{K}_z$ is $n,$ $\mathcal{K}_z$ is bounded.
Let $\{k_n\}$ be a sequence in $\mathcal{K}_z$ converging to some $k.$ It is evident that $k \geq0$ with $k(i,i)=1$ for all $i=1,\ldots,n$. Moreover, $k$ satisfies the conditions $$\begin{bmatrix} (1-z_i^{(2)}\bar{z}_j^{(2)})k(i,j)\end{bmatrix}_{i,j=1}^n\geq 0 \mbox{ and }\begin{bmatrix}(1-\psi_\alpha(z_i)\overline{\psi_\alpha(z_j)})k(i,j)\end{bmatrix}_{i,j=1}^n \geq 0, \quad \alpha\in \overline{\mathbb D}.$$ 
  If $k$ is not strictly positive definite, then there exists a non-zero vector $v=(v_1,\ldots,v_n)^t\in \mathbb C^{n\times 1}$ such that $kv=0.$ Let $A_1, A_2$ and $A_3$ be three $n\times n$ diagonal matrices with $(i,i)$th entries as $z_i^{(1)}, z_i^{(2)}$ and $z_i^{(3)}$, respectively, for $i=1, \ldots, n.$
    Thus, for any polynomial $p$ in three variables, $p(A_1,A_2,A_3)$ is a diagonal matrix with $(i,i)$th entry as $p(z_i).$ 

    Since $\begin{bmatrix}
        (1-z_i^{(2)}\bar{z}_j^{(2)})k(i,j)
    \end{bmatrix}_{i,j=1}^n\geq 0$,
   we have $0 \leq \inp{kv}{v}-\inp{kA_2v}{A_2v}.$ Which further yields  that $k(A_2v)=0.$  By taking $\alpha=0$ and $1$ in $\begin{bmatrix}(1-\psi_\alpha(z_i)\overline{\psi_\alpha(z_j)})k(i,j)\end{bmatrix}_{i,j=1}^n \geq 0,$ we get $k(A_1 v)=k(A_3v)=0.$
    Therefore $k(p(A_1,A_2,A_3)v)=0.$ Without loss of generality, assume that $v_j\neq 0$ for some $j \in \{1, \ldots, n\}.$ Choose a polynomial $q$ such that $q(z_i)=0$ for all $i\neq j$ and $q(z_j)=1.$ 
    Therefore, $k(q(A_1,A_2,A_3)v)=k (0,\ldots,v_j,\ldots,0)^t=0.$ Subsequently, the $j$th column of $k$ is a zero vector, which is a contradiction to $k(j,j)=1$. 

     Let $f\in \mathcal{H}\mathcal{E}(V)$ and assume that $g$ is an extremal function for the given data $Z=\{z_1,\ldots,z_n\}\subseteq V,\; W=\{w_1,\ldots,w_n\}\subseteq\mathbb C.$ For simplicity, let $\rho=\rho^f(Z,W).$ Since $\|g(T)\| \leq \|g\|_\mathcal{M}=\rho$ for any $T\in \mathcal{M}$, we have $\frac{1}{\rho}g\in \SAE$ and $\frac{1}{\rho}g(z_i)=\frac{w_i}{\rho}$ for all $i=1,\ldots,n.$ By Theorem \ref{interpolation rephrased}, $\begin{bmatrix}(\rho^2-w_i\bar{w}_j)k(i,j)\end{bmatrix}_{i,j=1}^n\geq 0$ for all $k\in \mathcal{K}_z.$

    Consider the set \beq \label{set 1} Y=\{\sigma\leq \|f\|_{\infty}: \begin{bmatrix}(\ \sigma^2-w_i \bar{w}_j)k(i,j)\end{bmatrix} \geq 0\; \text{for all } k\in \mathcal{K}_z\}.\eeq Since $\rho \in Y,$ $Y$ is nonempty. Now for any $\sigma \in Y,$ again by Theorem \ref{interpolation rephrased}, there exists $f_\sigma' \in \SAE$ such that $f_\sigma'(z_i)=\frac{w_i}{\sigma}$ for $i=1,\ldots,n.$ Setting $f_\sigma=\sigma f_\sigma'$, then for any $T\in \mathcal{M}$, $\|f_\sigma(T)\| \leq \sigma \leq \|f\|_{\infty}.$  
    Therefore, $f_\sigma\in \mathcal{SA}_f(\mathbb E)$ such that $f_\sigma(z_i)=w_i$ for all $i=1,\ldots,n$ and $\|f_\sigma\|_{\mathcal{M}}\leq \sigma.$ By definition of $\rho,$ $\rho \leq \|f_\sigma\|_\mathcal{M} \leq \sigma$. This shows that $\rho=\inf Y.$
    
 We now claim that there exist $k' \in \mathcal{K}_z$ and a non-zero column vector $x=(x_1,\ldots,x_n)^t \in \mathbb C^{n\times 1}$ such that \beq \label{extension eq1}\sum_{i,j=1}^n (\rho^2-w_i \bar{w}_j)k'(i,j)\bar{x}_i x_j = 0.\eeq
    Note that the set $\mathcal{C}=\{\tilde{k}=\begin{bmatrix}(\rho^2-w_i \bar{w}_j)k(i,j)\end{bmatrix}_{i,j=1}^n: k\in \mathcal{K}_z\}$ is compact. 
    Let $\lambda$ be the function on $\mathcal{C}$ which maps each element to its minimum eigenvalue.
    Clearly $\lambda$ is continuous and therefore, $\alpha=\inf\{\lambda(\tilde{k}): \tilde{k}\in \mathcal{C}\}$ is attained on $\mathcal{C}$.
    If $\alpha=0$, then we get a $k'\in \mathcal{K}_z$ and a non-zero vector $x=(x_1,\ldots,x_n)^t\in \mathbb C^{n\times 1}$ which satisfies \eqref{extension eq1}. To prove the claim, it suffices to show that $\alpha$ must be zero. On the contrary, assume $\alpha>0.$ Then for any $k\in \mathcal{K}_z$ and $x=(x_1,\ldots,x_n)^t \in \mathbb C^{n\times 1}$, we have $$\left \langle \begin{bmatrix}(\rho^2-w_i \bar{w}_j)k(i,j)\end{bmatrix} x,x\right\rangle \geq \lambda(k)\|x\|^2 \geq \alpha \|x\|^2.$$ 
   If $\beta=\rm{sup}\{\|k\|: k\in \mathcal{K}_z\},$ then by choosing $0 < \epsilon < \frac{\alpha}{\beta},$ we get $$\left \langle \begin{bmatrix}(\rho^2-\epsilon-w_i \bar{w}_j)k(i,j)\end{bmatrix} x,x\right\rangle \geq 0$$ for all $x\in C^{n}$ and $k \in \mathcal{K}_z$. This contradicts the fact that $\rho =\inf Y.$
    
  Let $k'(\cdot,j) \in \mathbb C^{n\times 1}$ denote the $j$th column vector of $k'\in \mathcal{K}_z$.
Define a $3$-tuple of operators $S=(A,B,P)$ on the finite dimensional Hilbert space spanned by $\{k'( \cdot, j), j=1, \ldots, n\}$ as follows : $$A^* k'(\cdot,j)=\bar{z}^{(1)}_j k'(\cdot,j);\;B^* k'(\cdot,j)=\bar{z}^{(2)}_j k'(\cdot,j);\;P^* k'(\cdot,j)=\bar{z}^{(3)}_j k'(\cdot,j)\; \text{for all } j=1, \ldots, n.$$ Since $k'$ is strictly positive definite,  $S^*$ is a well-defined commuting $3$-tuple of operators. Moreover, the Taylor joint spectrum $\sigma(S)=\{z_1,\ldots,z_n\}.$ In fact, $S$ is subordinate to $\{z_1,\ldots,z_n\}$ and $S,\;S^*\in \mathcal{M}'.$
   For the given extremal function $g$, we have $g(S)^*k'(\cdot,j)=\overline{g(z_j)} k'(\cdot, j)=\bar{w}_j k'(\cdot, j)$ for all $j=1,\ldots,n.$ 
    Now, $\begin{bmatrix}(\rho^2-w_i\overline{w}_j)k'(i,j)\end{bmatrix}_{i,j=1}^n \geq 0$ gives that $\|g(S)\|=\|g(S)^*\| \leq \rho$. By \eqref{extension eq1}, it follows that $\|g(S)\|=\rho.$ This completes the proof.
\end{proof}

\begin{proof}[Proof of the Theorem \ref{extension theorem}:]
Let $V\subseteq \mathbb E$ and $f \in \mathcal H \mathcal E(V)$ with the property \eqref{extension condition}. Consider a countable dense subset $\{z_i \in V :i=1, \ldots, N,\; N\in \mathbb N \cup\{\infty\}\}.$ For each $n\in \mathbb N$, let $\rho_n=\rho^f(Z_n, W_n)$, where $Z_n=\{z_1,\ldots,z_n\}$ and $W_n=\{f(z_1),\ldots,f(z_n)\}.$ By Lemma \ref{extremalfunction}, there is an extremal function $g_n$ such that $\rho_n=\|g_n\|_{\mathcal{M}}.$ It follows from Lemma \ref{lemma for extension}, there exists $S_n\in \mathcal{M}'$ subordinate to $Z_n$ such that $\|g_n(S_n)\|=\rho_n.$ Consequently, 
\beqn \|g_n\|_\mathcal{M} =\rho_n=\|g_n(S_n)\|
= \|f(S_n)\| \leq \|f\|_{\infty}.
\eeqn
Since $\|g_n\|_\infty \leq \|g_n\|_{\mathcal{M}}\leq \|f\|_{\infty}$, there is a subsequence $\{g_{n_k}\}$ of $\{g_n\}$ that converges pointwise to a function $g\in H^\infty(\mathbb E).$ Thus, $g|_{V}=f$ and $\|g\|_\infty \leq \|f\|_{\infty}.$ 
A similar use of the dominated convergence theorem, as in Lemma \ref{motivatingextension}, yields $\|g(T)\| \leq \|f\|_{\infty}$ for all $T\in \mathcal{M}.$ This completes the proof.
\end{proof}

For $V \subseteq\mathbb E,$ we say that $V$ has the extension property in $\SAE$ if for any $f\in \mathcal{H}\mathcal{E}(V),$ there exists $g \in H^\infty(\mathbb E)$ such that $\frac{1}{\|f\|_{\infty}}g\in \SAE$,\; $g|_V=f$ and $\|f\|_{\infty}=\|g\|_\infty.$ 
\begin{corollary}
    Let $V\subseteq \mathbb E$. Then $V$ has the extension property in $\SAE$ if and only if for every $f\in \mathcal{H}\mathcal{E}(V),$ $\|f(T)\| \leq \|f\|_{\infty}$ for all $T \in \mathcal{M}'$ subordinate to $V$.
\end{corollary}

\section{Toeplitz corona problem} \label{S5}
This section is devoted to establishing the Toeplitz corona theorem for the tetrablock. As a preparatory step, we first discuss the vector-valued counterpart of certain previously obtained results, including Theorem \ref{realization theorem}. However, we omit their proofs as it proceeds along the same line.
\subsection{Vector-valued realization theorem} \label{SS5.1} Let $\mathfrak{L}$ be a complex separable Hilbert space and $F \subset \mathbb E.$ A $\mathcal{B}(\mathfrak{L})$-\emph{valued weak kernel} on $F$ means a positive semi-definite function $k : F \times F \to \mathcal{B}(\mathfrak{L})$, that is,
$$
\sum_{i,j=1}^n \left\langle k(z_i, z_j) v_j, v_i \right\rangle_{\mathfrak{L}} \geq 0
$$ 
for every finite set $\{z_1, \ldots, z_n\} \subset F$ and for all vectors $v_1, \ldots, v_n \in \mathfrak{L}$. If, in addition, $k(z, z) \neq 0$ for all $z \in F$, then $k$ will be referred to as a $\mathcal{B}(\mathfrak{L})$-\emph{valued kernel}.

A  $\mathcal{B}(\mathfrak{L})$-valued kernel $k$ is said to be \emph{admissible} if $(1-z^{(2)}\overline{w}^{(2)})k(z,w)$ and $(1-\psi_\alpha(z)\overline{\psi_\alpha(w)})k(z,w)$ are positive semi-definite for all $\alpha \in \overline{\mathbb D}.$ Similarly, {\it weakly admissible} can be defined.
Further, a function $\Gamma:F \times F \to \mathcal{B}(C(\overline{\mathbb D}),\mathcal{B}(\mathfrak{L}))$ is called \emph{completely positive kernel} if for any $n\in \mathbb N$, $v_1,\ldots,v_n \in \mathfrak{L}$, $z_1,\ldots,z_n \in F$, and $h_1,\ldots,h_n \in C(\overline{\mathbb D})$, we have $$\sum_{i,j=1}^n \inp{\Gamma(z_i,z_j)(h_i \bar{h}_j)v_j}{v_i}_{\mathfrak{L}} \geq 0.$$

The following result serves as a vector-valued counterpart of Proposition \ref{decompositionresult} (cf. \cite[Lemma 3.3]{BS2022}).
\begin{proposition}\label{decompositionresultv}
    Let $\Gamma: F \times F \to \mathcal{B}(C(\overline{\mathbb D}), \mathcal{B}(\mathfrak{L}))$ be a completely positive kernel. Then there exist a Hilbert space $\mathfrak{H}$ and a function $L: F \to \mathcal{B}(C(\overline{\mathbb D}),\mathcal{B}(\mathfrak{H},\mathfrak{L}))$ such that $$\Gamma(z,w)(h_1 \bar{h}_2)=L(z)(h_1)L(w)(h_2)^*\; \text{for all } z,w\in F\; \text{and } h_1,h_2\in C(\overline{\mathbb D}).$$ Moreover, there is a unital $*$-representation $\rho : C(\overline{\mathbb D}) \to \mathcal{B}(\mathfrak{H})$ such that $$\rho(h_1)^*L(z)(h_2)^*=L(z)(h_1 h_2)^*\; \text{for all } z \in F \; \text{and}\; h_1,h_2 \in C(\overline{\mathbb D}).$$
\end{proposition}

 Following the notation of \cite[Definition 11.25]{AM2002}, for functions $k_1,k_2 : \mathbb{E} \times \mathbb{E} \to \mathcal{B}(\mathfrak{L}),$ we define  $k_1 \oslash k_2: \mathbb{E} \times \mathbb{E} \to \mathcal{B}(\mathfrak{L} \otimes \mathfrak{L})$ to be the function given by
\[
(k_1 \oslash k_2)(z,w) = k_1(z,w) \otimes k_2(z,w), \quad z,w \in \mathbb{E}.
\]

The following result provides a vector-valued analog of Theorem \ref{keytheorem} (cf. \cite[Theorem 11.26]{AM2002}, \cite[Lemma 5.1]{BS2022}).
\begin{theorem} \label{keytheoremv}
     Let $g: \mathbb E \times \mathbb E\to \mathcal{B}(\mathfrak{L})$ be a self-adjoint function, that is, $g(z,w)=g(w,z)^*$. Then $g \oslash k$ is positive semi-definite for any $\mathcal{B}(\mathfrak{L})$-valued admissible kernel $k$ if and only if there exist a completely positive kernel $\Gamma:\mathbb E \times \mathbb E \to\mathcal{B}(C(\overline{\mathbb D}),\mathcal{B}(\mathfrak{L}))$ and a $\mathcal{B}(\mathfrak{L})$-valued weak kernel $\Delta$ on $\mathbb E$ such that $$g(z,w)=\Gamma(z,w)(1-E(z)\overline{E(w)})+(1-z^{(2)}\overline{w}^{(2)})\Delta(z,w),\; z,w\in \mathbb E.$$
\end{theorem}
 We now present the realization theorem (Theorem \ref{realization theorem}) in the vector-valued setting.  
\begin{theorem}\label{realization theorem vector}
    Given complex separable Hilbert spaces $\mathfrak{L}_1$ and $\mathfrak{L}_2,$ let $f : \mathbb E \to \mathcal{B}(\mathfrak{L}_1,\mathfrak{L}_2)$ be a function. Then the following statements are equivalent:
    \begin{itemize}
        \item [(i)] $f \in \mathcal{SA}_{\mathbb E}(\mathfrak{L}_1,\mathfrak{L}_2).$
        \item [(ii)]  $(I_{\mathfrak{L}_2} - f(z) f (w)^*)\oslash k(z, w)$ is a positive semi-definite function for all $\mathcal{B}(\mathfrak{L}_2)$-valued admissible kernel $k$ on $\mathbb E$.
         \item [(iii)] There exist a completely positive kernel $\Gamma: \mathbb E \times \mathbb E \rightarrow \mathcal{B}(C(\mathbb D), \mathcal{B}(\mathfrak{L}_2))$ and a $\mathcal{B}(\mathfrak{L}_2)$-valued weak kernel $\Delta$ on  $\mathbb E$ such that 
        \begin{equation*}
          I_{\mathfrak{L}_2} - f(z) f(w)^*=\Gamma(z, w)\Big(1- E(z)\overline{E(w)}\Big)+ (1-z^{(2)}\overline{w}^{(2)})\Delta(z,w),\; z,w \in \mathbb E.
        \end{equation*} 
         \item [(iv)] There are Hilbert spaces $\mathfrak{H}_1,\;\mathfrak{H}_2$, a unital $*$-representation $\rho :C(\overline{\mathbb D}) \rightarrow \mathcal{B}(\mathfrak{H}_1)$ and a unitary $V : \mathfrak{L}_2 \oplus \mathfrak{H} \rightarrow \mathfrak{L}_1 \oplus \mathfrak{H},$ where $\mathfrak{H}=\mathfrak{H}_1\oplus \mathfrak{H}_2,$ such that $$V=\begin{bmatrix}
A_1 & B_1 \\
C_1 & D_1
\end{bmatrix}, \; \text{and}$$
 \beqn \;\;\;\;\; f(z)^*= A_1 +B_1\begin{bmatrix}
    \rho(E(z))^* & 0\\ 0& \bar{z}^{(2)}I_{\mathfrak{H}_2}
\end{bmatrix}\left(I_\mathfrak{H}-D_1\begin{bmatrix}
    \rho(E(z))^* & 0\\ 0& \bar{z}^{(2)}I_{\mathfrak{H}_2}
\end{bmatrix}\right)^{-1}C_1.\eeqn
    \end{itemize}
\end{theorem}

\begin{remark}
    It is easy to see that part $\mathrm{(iv)}$ of Theorem \ref{realization theorem vector} is equivalent to saying that there exist Hilbert spaces $\mathfrak{H}_1,\;\mathfrak{H}_2$,  a unital $*$-representation $\rho :C(\overline{\mathbb D}) \rightarrow \mathcal{B}(\mathfrak{H}_1)$ and a unitary $U : \mathfrak{L}_1 \oplus \mathfrak{H} \rightarrow \mathfrak{L}_2 \oplus \mathfrak{H}, \; \mathfrak{H}=\mathfrak{H}_1\oplus \mathfrak{H}_2,$ such that $$U=\begin{bmatrix}
A & B \\
C & D
\end{bmatrix},\; \text{and}$$
we have $$\,\,\,\,\,\,\,\,\,\,f(z)=A +B\begin{bmatrix}
    \rho(E(z)) & 0\\ 0& z^{(2)}I_{\mathfrak{H}_2}
\end{bmatrix}\left(I_\mathfrak{H}-D\begin{bmatrix}
    \rho(E(z)) & 0\\ 0& z^{(2)}I_{\mathfrak{H}_2}
\end{bmatrix}\right)^{-1}C.$$ 
\end{remark}

As an application of Theorem \ref{realization theorem vector}, we give the vector-valued counterpart of the interpolation theorem (cf. \cite[Theorem 11.49]{AM2002}).
\begin{corollary}
    Let $\mathfrak{L}_1,\mathfrak{L}_2$ be two Hilbert spaces. Let $z_1,\ldots,z_n \in \mathbb E$ and $W_1,\ldots,W_n$ be the operators in $\mathcal{B}(\mathfrak{L}_1,\mathfrak{L}_2)$. Then the following are equivalent.
    \begin{itemize}
        \item [(i)] There is a function $f \in \mathcal{SA}_{\mathbb E}(\mathfrak{L}_1,\mathfrak{L}_2)$ such that $f(z_i)=W_i$ for all $i=1,\ldots,n$.
        \item [(ii)] For each $\mathcal{B}(\mathfrak{L}_2)$-valued admissible kernel $k$, we have $$\begin{bmatrix}(I_{\mathfrak{L}_2}-W_i W_j^*)\otimes k(z_i,z_j)\end{bmatrix}_{i,j=1}^n \geq 0.$$ 
        \item [(iii)]There exist a completely positive kernel $\Gamma: F \times F \to \mathcal{B}(C(\overline{\mathbb D}),\mathcal{B}(\mathfrak{L}_2))$ and a $\mathcal{B}(\mathfrak{L}_2)$-valued weak kernel $\Delta $ on $F$, where $F=\{z_1,\ldots,z_n\}$, such that $$I_{\mathfrak{L}_2}-W_iW_j^*=\Gamma(z_i,z_j)(1-E(z_i)\overline{E(z_j)})+ (1-z_i^{(2)}\bar{z}_j^{(2)})\Delta(z_i,z_j),\; 1 \leq i,j \leq n.$$
    \end{itemize}
\end{corollary}

\subsection{Toeplitz corona theorem} \label{SS5.2}
The classical corona problem asks for conditions on given functions $\varphi_1,\ldots,\varphi_n \in H^\infty(\mathbb D)$ under which there exist functions $ \psi_1,\ldots,\psi_n \in H^\infty(\mathbb D)$ satisfying $$\varphi_1 \psi_1 +\cdots+ \varphi_n\psi_n=1.$$ 

Carleson's theorem in \cite{CL1962} assures that the condition $\sum_{j=1}^n |\varphi_j|^2 \geq \delta$ for some $\delta>0,$ is necessary and sufficient. Arveson gave an operator-theoretic analogue of this result in \cite[Theorem 6.3]{AV1975}, commonly referred to as the Toeplitz corona theorem. Later, Schubert provided a short proof of the Toeplitz corona theorem in \cite{SB1978}. In fact, Schubert shows that for given $\varphi_1,\ldots,\varphi_n \in H^\infty(\mathbb D)$ and $\delta>0$, there exist $\psi_1,\ldots,\psi_n \in H^\infty(\mathbb D)$ with the property $\varphi_1 \psi_1 +\cdots+ \varphi_n\psi_n=1$ and $\sup_{z\in \mathbb D}\{|\psi_1(z)|^2+\cdots+|\psi_n(z)|^2\} \leq \frac{1}{\delta^2}$ (in other words, $(\psi_1,\ldots,\psi_n)^t\in \frac{1}{\delta}\mathcal{S}_\mathbb D(\mathbb C,\mathbb C^n)$) if and only if
\beq \label{coronaclassic} T_{\varphi_1} T_{\varphi_1}^*+\cdots+T_{\varphi_n} T_{\varphi_n}^* \geq \delta^2 I > 0, \eeq where $T_{\varphi_i}$ is Toeplitz operator with symbol $\varphi_i$ defined on the Hardy space $H^2(\mathbb D)$ for each $i=1, \ldots, n.$ For the unit polydisc and the unit ball, the Toeplitz corona theorem was obtained by E. Amar in \cite[Theorem 1.1]{Am2003}.

Note that the condition in \eqref{coronaclassic} is equivalent to the positive semi-definiteness of the product of $\sum_{j=1}^n \varphi_i(z)\overline{\varphi_i(w)}-\delta^2$ and the Szeg\"o kernel of the unit disc.
In the case of bidisc, the criterion involved a family of kernels instead of the Szeg\"o kernel. 
Below, we provide an analog of the Toeplitz corona theorem for the domain $\mathbb E,$ in fact, we prove a slightly more general version of this result. The proof involves a familiar technique that is used to prove \cite[Theorem 8.57, Theorem 11.65]{AM2002} and \cite[Theorem 5.2]{BS2022} for the unit disc, unit bidisc, and symmetrized bidisc, respectively. 
\begin{theorem} \label{vectorTC}
    Let $\mathfrak{L}_1,\mathfrak{L}_2$ and $\mathfrak{L}_3$ be complex separable Hilbert spaces. Suppose $\Phi: \mathbb E \rightarrow \mathcal{B}(\mathfrak{L}_1,\mathfrak{L}_2)$ and $\Theta: \mathbb E \rightarrow \mathcal{B}(\mathfrak{L}_3,\mathfrak{L}_2)$ are given functions. Then the following assertions are equivalent:
    \begin{itemize}
        \item [(i)] There exists a function $\Psi \in \mathcal{SA}_{\mathbb E}(\mathfrak{L}_3,\mathfrak{L}_1)$ such that $$\Phi(z)\Psi(z)=\Theta(z)\;\; \text{for all } z \in \mathbb E.$$
        \item [(ii)] For any finite set $F =\{z_1,\ldots,z_n\}\subseteq \mathbb E$, \beq \label{coronatoeplitz eq1}\begin{bmatrix}(\Phi(z_i)\Phi(z_j)^*-\Theta(z_i)\Theta(z_j)^*) \otimes k(z_i,z_j)\end{bmatrix}_{i,j=1}^n \geq 0
        \eeq for all $\mathcal{B}(\mathfrak{L}_2)$-valued admissible kernel $k.$
         \item [(iii)] There exist a completely positive kernel $\Gamma: \mathbb E \times \mathbb E \rightarrow \mathcal{B}(C(\mathbb D), \mathcal{B}(\mathfrak{L}_2))$ and a $\mathcal{B}(\mathfrak{L}_2)$-valued weak kernel $\Delta$ on $\mathbb E$ such that for all $z, w \in \mathbb E,$ \beq \label{coronatoeplitz eq2}\Phi(z)\Phi(w)^*-\Theta(z)\Theta(w)^*= \Gamma(z,w)(1-E(z)\overline{E(w)})+ (1-z^{2}\overline{w}^{(2)})\Delta(z,w).\eeq
    \end{itemize}
\end{theorem}

\begin{proof}
    $\mathrm{(i)} \implies \mathrm{(ii)}:$ To prove this implication, we use the similar techniques that proves $\rm(i)$ implies $\rm{(ii)}$ in Theorem \ref{realization theorem}.
   Let $k$ be a $\mathcal{B}(\mathfrak{L}_2)$-valued admissible kernel on $\mathbb E.$ 
    Then there exist a Hilbert space $\mathfrak{H}$ and a function $Q :\mathbb E \to \mathcal{B}(\mathfrak{L}_2, \mathfrak{H})$ such that $k(z,w)=Q(z)^*Q(w).$ 
    Let $H_n(k)=\bigvee\{Q(z_i)u: i=1,\ldots,n,\; u\in \mathfrak{L}_2\}.$
    Define the operators $T_i$ on $H_n(k)$ as $T_i(Q(z_j)u)=\bar{z}_j^{(i)} Q(z_j)u$ for $i=1,2,3,\; j=1,\ldots, n.$ Since $k$ is admissible, $\|T_2\| \leq 1$ and $\|\psi_\alpha(T)\| \leq 1$ for all $\alpha \in \overline{\mathbb D}$. In addition to hypothesis, assume that $\Psi\in A_{\overline{\mathbb E}}(\mathfrak{L}_3,\mathfrak{L}_1).$ A vector-valued version of the Lemma \ref{hat}, gives that $\hat{\Psi}(z)=\Psi(\bar{z})^*$ is in $\mathcal{SA}_{\mathbb{E}}(\mathfrak{L}_1,\mathfrak{L}_3) \cap A_{\overline{\mathbb E}}(\mathfrak{L}_1,\mathfrak{L}_3).$ 
    For any $0 <s<1$, the map $\hat{\Psi}(s \cdot T): H_n(k) \otimes \mathfrak{L}_1 \to H_n(k) \otimes \mathfrak{L}_3$ is given by 
   $$\hat{\Psi}(s \cdot T) (Q(z_i)u\otimes v)=Q(z_i)u \otimes \Psi(s \cdot z_i)^* v,\; i=1,\ldots,n, u \in \mathfrak{L}_2, v\in \mathfrak{L}_1.$$ 
    For $z_1,\dots, z_n \in \mathbb E, u_1, \ldots, u_n \in \mathfrak{L}_2$ and $v_1,\ldots v_n \in \mathfrak{L}_1,$    \begin{align*}
          \sum_{i,j=1}^n \left \langle \left( I_{H_n(k)\otimes \mathfrak{L}_1}-\hat{\Psi}(s \cdot T)^* \hat{\Psi}(s \cdot T) \right)(Q(z_j)u_j\otimes \Phi(z_j)^*v_j), Q(z_i)u_i\otimes \Phi(z_i)^*v_i \right \rangle \geq 0.
           \end{align*}
This yields that
\beqn
\sum_{i,j=1}^n \left \langle \left(\Phi(z_i)\Phi(z_j)^*-\Phi(z_i)\Psi(s \cdot z_i)\Psi(s \cdot z_j)^*\Phi(z_j)^*\right ) v_j,v_i\right \rangle \inp{k(z_i,z_j)u_j}{u_i} \geq 0.
\eeqn 
By taking $s\to 1^{-}$ we get $$\sum_{i,j=1}^n \left \langle \left(\Phi(z_i)\Phi(z_j)^*-\Phi(z_i)\Psi(z_i)\Psi(z_j)^*\Phi(z_j)^*\right ) v_j,v_i\right \rangle \inp{k(z_i,z_j)u_j}{u_i} \geq 0.$$ 
For $\Psi\in \mathcal{SA}_\mathbb E(\mathfrak{L}_3,\mathfrak{L}_1)$, consider $\Psi_r:\overline{\mathbb E}\to \mathcal{B}(\mathfrak{L}_3,\mathfrak{L}_1)$ by $\Psi_r(z)=\Psi(r \cdot z)$ for any $0 <r <1.$ 
Thus, $\Psi_r \in \mathcal{SA}_\mathbb E(\mathfrak{L}_3,\mathfrak{L}_1) \cap A_{\overline{\mathbb E}}(\mathfrak{L}_3,\mathfrak{L}_1).$ We complete the proof by taking $r \to 1^-.$ 

    $\mathrm{(ii)} \implies \mathrm{(iii)}:$ It follows from Theorem \ref{keytheoremv}.
    
$\mathrm{(iii)} \implies \mathrm{(i)}:$ Here, we invoke the technique of lurking isometry already used in the proof of the realization theorem. Suppose the condition in \eqref{coronatoeplitz eq2} holds true.
  By Proposition \ref{decompositionresultv}, there exist a Hilbert space $\mathfrak{H}_1$ and a function $L: \mathbb E \to \mathcal{B}(C(\overline{\mathbb D}), \mathcal{B}(\mathfrak{H}_1,\mathfrak{L}_2))$ such that $\Gamma(z,w)(h_1 \bar{h}_2)=L(z)(h_1)L(w)(h_2)^*$ for all $z,w\in \mathbb E$ and $h_1,h_2 \in C(\overline{\mathbb D}).$ Moreover, there exists a unital $*$-representation $\rho: C(\overline{\mathbb D}) \to \mathcal{B}(\mathfrak{H}_1).$ By Theorem 2.62 of \cite{AM2002}, there exist a Hilbert space $\mathfrak{H}_2$ and a function $G:\mathbb E \to \mathcal{B}(\mathfrak{H}_2, \mathfrak{L}_2)$ such that $\Delta(z,w)=G(z)G(w)^*$. 
  Therefore, \eqref{coronatoeplitz eq2} can be re-written as
     \begin{align*}
      & \Phi(z)\Phi(w)^*+ L(z)(E(z))L(w)(E(w))^* + z^{(2)}\overline{w}^{(2)}G(z)G(w)^*\\
      &= \Theta(z)\Theta(w)^* +L(z)(1)L(w)(1)^* + G(z)G(w)^*.
  \end{align*}
  Consider $\mathfrak{H}=\mathfrak{H}_1 \oplus\mathfrak{H}_2$ and subspaces given by 
  \beqn
  \mathfrak{H}_d & =& \mathrm{span}\{ \Phi(z)^* v\oplus \rho(E(z))^* L(z)(1)^* v \oplus \bar{z}^{(2)}G(z)^*v: v \in \mathfrak{L}_2, z\in \mathbb E\} \subseteq \mathfrak{L}_1 \oplus \mathfrak{H},\\
  \mathfrak{H}_r &=& \mathrm{span}\{ \Theta(z)^* v \oplus L(z)(1)^* v \oplus G(z)^* v: z\in \mathbb E, v\in \mathfrak{L}_2\} \subseteq \mathfrak{L}_3 \oplus \mathfrak{H}.
  \eeqn 
    The map $V: \mathfrak{H}_d \to \mathfrak{H}_r$ defined by 
    $$V \begin{pmatrix}
        \Phi(z)^* v \\ \rho(E(z))^* L(z)(1)^* v\\ \bar{z}^{(2)}G(z)^*v
    \end{pmatrix} =V \begin{pmatrix}
        \Phi(z)^* v\\ {\begin{bmatrix}
        \rho(E(z))^* & 0\\ 0 & \bar{z}^{(2)} I_{\mathcal{H}_2}
    \end{bmatrix} \begin{pmatrix}
        L(z)(1)^*v \\ G(z)^*v
    \end{pmatrix}}
    \end{pmatrix}=\begin{pmatrix}
        \Theta(z)^*v\\ L(z)(1)^* v \\ G(z)^*v
    \end{pmatrix}$$
    is an isometry. We now extend $V$ to an unitary operator from $\mathfrak{L}_1 \oplus\mathfrak{H}$ onto $\mathfrak{L}_3 \oplus \mathfrak{H}$ and write $V=\begin{bmatrix}
       A_1  & B_1\\ C_1 & D_1
    \end{bmatrix}.$
    Set $X(z)= \begin{bmatrix}
        \rho(E(z)) &0\\ 0 & z^{(2)} I_{\mathfrak{H}_2} \end{bmatrix}.$ 
        Therefore, we have
    \beqn \label{coronatoeplitz eq3}A_1 \Phi(z)^*v + B_1 X(z)^* \left(\substack{L(z)(1)^*v \\ G(z)^*v}\right)= \Theta(z)^*v,\\
   \label{coronatoeplitz eq4} C_1 \Phi(z)^*v + D_1 X(z)^* \left(\substack{L(z)(1)^*v \\ G(z)^*v}\right)= \left(\substack{L(z)(1)^*v \\ G(z)^*v}\right).
    \eeqn 
    By combining above two equations, we get $$\Theta(z)^*v= A_1 \Phi(z)^*v + B_1 X(z)^*(I_\mathfrak{H}-D_1X(z)^*)^{-1}C_1 \Phi(z)^*v, \; z \in \mathbb E, \; v\in \mathfrak{L}_2.$$ 
    Therefore, it is natural to define $\Psi:\mathbb E \to \mathcal{B}(\mathfrak{L}_3,\mathfrak{L}_1)$ in such that $$\Psi(z)^*= A_1 +B_1 X(z)^* (I_\mathfrak{H}-D_1 X(z)^*)^{-1} C_1\in \mathcal{B}(\mathfrak{L}_1,\mathfrak{L}_3).$$ By Theorem \ref{realization theorem vector}(iv), $\Psi \in \mathcal{SA}_\mathbb E(\mathfrak{L}_3,\mathfrak{L}_1).$ Moreover, $\Theta(z)=\Phi(z)\Psi(z).$ This completes the proof.
\end{proof}
The proof of Theorem \ref{Toeplitzcoronatheorem1} follows by taking $\mathfrak{L}_3=\mathfrak{L}_2=\mathbb C$, and $\mathfrak{L}_1=\mathbb C^n$ in Theorem \ref{vectorTC}.
\subsection*{Acknowledgments}
The author, S. Jain, gratefully acknowledges Prof. Arup Chattopadhyay for the financial support provided through his project (Ref No. MATHSPNSERB01119xARC002).
The work of S. Kumar was supported in the form of MATRICS grant (Ref No. MTR/2022/000457) of Science and Engineering Research Board (SERB). 
Support for the work of M. K. Mal was provided in the form of a Prime Minister's Research Fellowship (PMRF / 2502827). 
Support for the work of P. Pramanick was provided by the Department of Science and Technology (DST) in the form of the Inspire Faculty Fellowship (Ref No. DST/INSPIRE/04/2023/001530).

\end{document}